\documentclass{cmslatex}

\usepackage[paperwidth=7in, paperheight=10in, margin=.875in]{geometry}
\usepackage[backref,colorlinks,linkcolor=red,anchorcolor=green,citecolor=blue]{hyperref}
\usepackage{amsfonts,amssymb}
\usepackage{amsmath}
\usepackage{graphicx}
\usepackage{cite}
\usepackage{enumerate}
\usepackage{multirow}
\usepackage{epstopdf}
\usepackage{float}
\usepackage{subfigure}
\allowdisplaybreaks
\renewcommand{\theequation}{\arabic{section}.\arabic{equation}}

\allowdisplaybreaks

\newcommand{\yu}[1]{{{#1}}}
\newcommand{\new}[1]{{#1}}

\begin{document}
\title{Energy Stable Second Order Linear Schemes for the
  Allen-Cahn Phase-Field Equation\thanks{Feb 14, 2018, and
    accepted date (The correct dates will be entered by the
    editor).}}

\author{
  Lin Wang\thanks{LSEC, Institute of Computational
    Mathematics and Scientific/Engineering Computing,
    Academy of Mathematics and Systems Science, Beijing
    100190, China; School of Mathematical Sciences,
    University of Chinese Academy of Sciences, Beijing
    100049, China (wanglin@lsec.cc.ac.cn).}
  \and Haijun Yu \thanks{NCMIS \& LSEC, Institute of
    Computational Mathematics and Scientific/Engineering
    Computing, Academy of Mathematics and Systems Science,
    Beijing 100190, China; School of Mathematical Sciences,
    University of Chinese Academy of Sciences, Beijing
    100049, China (hyu@lsec.cc.ac.cn).}
}

\pagestyle{myheadings} \markboth{Energy Stable Second Order
  Linear Schemes for Allen-Cahn Equation}{L. Wang and H. Yu}

\maketitle

\begin{abstract}

  Phase-field model is a powerful mathematical tool to
  study the dynamics of interface and morphology changes
  in fluid mechanics and material sciences. However,
  numerically solving a phase field model for a real problem
  is a challenge task due to the non-convexity of the bulk
  energy and the small interface thickness parameter in the
  equation. In this paper, we propose two stabilized second
  order semi-implicit linear schemes for the Allen-Cahn
  phase-field equation based on backward differentiation
  formula and Crank-Nicolson method, respectively. In both
  schemes, the nonlinear bulk force is treated explicitly
  with two second-order stabilization terms, which make the
  schemes unconditional energy stable and numerically
  efficient.  By using a known result of the spectrum estimate
  of the linearized Allen-Cahn operator and some regularity estimates of the exact
  solution, we obtain an
  optimal second order convergence in time with a prefactor
  depending on the inverse of the characteristic interface
  thickness only in some lower polynomial order. Both
  2-dimensional and 3-dimensional numerical results are
  presented to verify the accuracy and efficiency of
  proposed schemes.

\end{abstract}



\section{Introduction}\label{intro}

In this paper, we consider numerical approximation for the
Allen-Cahn equation with Neumann boundary condition
\begin{alignat}{2}
\label{eq:AC}
    &\phi_{t} =\gamma\Big{(}\varepsilon \Delta \phi - \dfrac{1}{\varepsilon}f(\phi)\Big{)},
      \quad&(x,t)\in \Omega \times (0,T], \\
    &\partial_{n}\phi =0, \quad &(x,t)\in \partial\Omega\times(0,T]. \label{eq:CH:nbc}
\end{alignat}
Here $\Omega \in R^{d}, d=2,3$ is a bounded domain with a
locally Lipschitz boundary, $n$ is the outward normal, $T$
is a given time, $\phi(x,t)$ is the phase-field variable.
$f(\phi)$, the bulk force, is the derivative of a given
energy function $F(\phi)$, which is usually non-convex with
two or more than two local minima. One commonly used energy
function for two-phase problem is the double-well potential
$F(\phi)=\frac{1}{4}(\phi^2-1)^2$.
$\varepsilon$ is the thickness of the interface between two
phases.  $\gamma$, called mobility, is related to the
characteristic relaxation time of the system.  The
homogeneous Neumann boundary condition implies that no mass
loss occurs across the boundary walls.  The equation
\eqref{eq:AC} is introduced by Allen and Cahn
\cite{allen_microscopic_1979} to describe the process of
phase separation in multi-component alloy systems.  It can
be regarded as the $L^2$ gradient flow with respect to the
Ginzburg-Landau energy functional
\begin{equation}\label{eq:AC:E}
  E_{\varepsilon}(\phi):=\int_{\Omega}\Big(\frac{\varepsilon}{2}|\nabla \phi|^{2}
  +\frac{1}{\varepsilon} F(\phi)\Big)dx.
\end{equation}
The corresponding energy dissipation is given as
\begin{equation}\label{eq:AC:E1}
  \frac{d}{dt} E_{\varepsilon}(\phi)=-\frac{1}{\gamma}\int_{\Omega}\|\phi_t\|^2dx \leq 0.
\end{equation}
Another popular phase field model is the Cahn-Hilliard equation, which is the
the $H^{-1}$ gradient flow with respect to the Ginzburg-Landau
energy functional. 
It was
originally introduced by Cahn and
Hilliard~\cite{cahn_free_1958} to describe the phase
separation and coarsening phenomena in non-uniform systems
such as alloys, glasses and polymer mixtures.  

The Allen-Cahn equation and the Cahn-Hilliard equation are
widely used in modeling many interface problems due to their
good mathematical properties
(cf. e.g. \cite{elliott_cahn-hilliard_1989,
  chen_spectrum_1994, elliott_cahnhilliard_1996,
  elder_modeling_2004, yue_diffuse-interface_2004,
  wu_stabilized_2014} ). However, the small parameter
$\varepsilon$ and the non-convexity of energy function $F$
make the numerical approximation of a phase field equation a
challenging task, especially the design of time marching
schemes.
It is well-known that if a fully explicit or implicit time
marching scheme is used, a tiny time step-size is required
for the semi-discretized scheme to be stable or uniquely solvable
since the nonlinear function $F$ is neither convex nor
concave.
A very popular approach to obtain unconditional stable time
marching schemes is the so called convex splitting method
which appears to be introduced by Elliott and
Stuart~\cite{elliott_global_1993}, and popularized by
Eyre~\cite{eyre_unconditionally_1998}, in which, the convex
part of $F(\phi)$ is treated implicitly and the concave part
of $F(\phi)$ is treated explicitly.  This method has been
applied to various gradient flows (see
e.g. \cite{baskaran_energy_2013, shen_numerical_2010, feng_analysis_2015, feng_analysis_2016, feng_uniquely_2018}). 
\yu{Traditional convex splitting schemes are first order accurate.
	Recently, several extensions to second order schemes were proposed
	based on either the Crank-Nicolson scheme (see e.g.\cite{baskaran_energy_2013, chen_linear_2014, guo_h2_2016, diegel_stability_2016, cheng_second-order_2016, li_second-order_2017}), 
	or second order backward differentiation formula (BDF2) \cite{yan_second-order_2018, li_second_2018}.
In all convex splitting schemes, no matter first order or second order, one usually obtains an uniquely solvable nonlinear convex problem at each time step.}


There are another types of second order unconditional stable
schemes for the phase field equations. In
\cite{du_numerical_1991}, Du and Nicolaides proposed a
secant-line method which is energy stable and second order
accurate. It is used and extended in several other works,
e.g. \cite{furihata_stable_2001, feng_fully_2006,
  condette_spectral_2011,gomez_provably_2011,
  baskaran_energy_2013,zhang_adaptive_2013,
  benesova_implicit_2014}.  Similar to the convex splitting
method, the secant-line method leads to nonlinear
semi-discretized system, which need special efforts to
solve.  Recently, an augmented Lagrange multiplier(ALM)
method was proposed in
\cite{guillen-gonzalez_linear_2013,guillen-gonzalez_second_2014}
to get second order linear energy stable schemes.  The idea
is generalized as invariant energy quadratization (IEQ) by
Yang et al. and successfully applied to handle several very
complicated nonlinear phase-field models (see
e.g. \cite{yang_linear_2016, han_numerical_2017,
  yang_efficient_2017,yang_yu_efficient_2017}). Based on
similar methodology, a new variant called scalar auxiliary
variable (SAV) method is developed by Shen et al.
\cite{shen_efficient_2017, shen_SAV_2017}. In the ALM and
IEQ approach, nonlinear semi-discretized systems are
avoided, but one has to solve variable-coefficient systems,
while in the SAV scheme, one only need to solve some linear
systems with constant coefficients. Different to other
methods, the energy in ALM, IEQ and SAV approach is a
modified one which also depends on the auxiliary variable.

In this study, we focus on numerical methods that based on semi-implicit 
discretization 
and
stabilization skill. To improve the numerical stability of
solving phase-field equations, semi-implicit schemes were proposed by Chen and Shen\cite{chen_applications_1998}
and Zhu et al.\cite{zhu_coarsening_1999}.  Although not
unconditionally stable, semi-implicit schemes allow much
larger time step-sizes than explicit schemes.  To further
improve the stability, Xu and Tang proposed stabilized
semi-implicit methods for epitaxial growth model in
\cite{xu_stability_2006}. The proposed schemes have
extraordinary numerical stability even though the
mathematical proof of the stability is not complete. Similar
schemes was developed for phase field equation by He et
al.\cite{he_large_2007} and Shen and
Yang\cite{shen_numerical_2010}, where the latter one adopted
a mixed form for the Cahn-Hilliard equation, by using a
truncated double well potential such that the assumption
$\|f'(\phi)\|_{\infty} \leq L$ is satisfied, the
unconditional energy stability was proved for the first order
stabilized scheme.
\yu{It is worth to mention that with no truncation made to
$f(\phi)$, Li et al \cite{li_characterizing_2016,
  li_second_2017} proved that the energy stable property can
be obtained as well, but a much larger stability constant
need be used.  }

In this paper, we develop two second-order unconditionally
energy stable linear schemes for the Allen-Cahn equation
based on the schemes proposed in \cite{xu_stability_2006}
and \cite{shen_numerical_2010}. The energy dissipation is
guaranteed by including two second order stabilization
terms, the first one is directly from
\cite{xu_stability_2006}, the other one is inspired by the
work \cite{wu_stabilized_2014}. We also carry out an optimal error
estimate for the time semi-discretized schemes.  For the
phase field equations, the error bounds will depend on the
factor of $1/\varepsilon$ exponentially if one uses a
standard procedure. By using a spectrum estimate result of
de Mottoni and Schatzman \cite{de_mottoni_evolution_1989,
  de_mottoni_geometrical_1995} and
Chen\cite{chen_spectrum_1994} for the linearized Allen-Cahn
operator, we are able to get an optimal error estimate with
a prefactor depend on $1/\varepsilon$ only in some lower
polynomial order for small $\varepsilon$.  This spectrum
estimate argument was first used by Feng and Prohl
\cite{feng_numerical_2003, feng_error_2004} for an implicit
first order scheme for phase field equations.  It was also
applied by Kessler et al.~\cite{kessler_posteriori_2004} to
derive a posteriori error estimate for adaptive time
marching.  Similar analysis for a first-order stabilized
semi-implicit scheme of the Allen-Cahn equation in given by
Yang~\cite{yang_error_2009}. \yu{Recently, Feng and Li \cite{feng_analysis_2015}
, Feng et al.~\cite{feng_analysis_2016} extended this spectrum estimate argument 
 to first order convex splitting scheme coupled with interior penalty discontinuous Galerkin spatial discretization for Allen-Cahn and Cahn-Hilliard equation, respectively.} To our best knowledge, our
analysis is the first such result for second order linear
schemes.  In summary, the proposed methods have several
merits: 1) They are second order accurate; 2) They lead to
linear systems with constant coefficients after time
discretization; 3) The stability and error analysis bases on
weak formulations, so both finite element method and
spectral method can be used for spatial discretization to
satisfy discretized energy dissipation law. 4) The methods can be easily used in more complicated systems. Note that,
similar approach can be extended to the Cahn-Hilliard
equation \cite{WangYu2017b, wang_two_2017}, where Lipschitz
condition of $f$ is assumed based on physical intuition and
the analyses are more tedious.


The remain parts of the paper is organized as follows. In
Section 2, we present the two second-order stabilized
schemes for the Allen-Cahn equation and prove they are
energy stable.  The error estimate to derive a convergence
rate that does not depend on $1/\varepsilon$ exponentially
is then constructed in Section 3.  Detailed implementation
and numerical experiments for problems in both 2-dimensional
and 3-dimensional tensor-product domain are presented in
Section 4 to verify our theoretical results.  We end the
paper with some conclusions in Section 5.

\section{The two second order stabilized linear
  schemes}\label{Sec.2}

We first introduce some notations which will be used
throughout the paper. We use $\|\cdot\|_{m,p}$ to denote the
standard norm of the Sobolev space $W^{m,p}(\Omega)$. In
particular, we use $\|\cdot\|_{L^p}$ to denote the norm of
$W^{0,p}(\Omega)=L^{p}(\Omega)$; $\|\cdot\|_{{m}}$ to denote
the norm of $W^{m,2}(\Omega)=H^{m}(\Omega)$; and $\|\cdot\|$
to denote the norm of $W^{0,2}(\Omega)=L^{2}(\Omega)$.  Let
$(\cdot, \cdot)$ represent the $L^{2}$ inner product. For $p\ge 0$, we define
$H_0^p(\Omega) := \{\,u\in H^p(\Omega)\,|\, (u,1)=0 \,\}$, and 
denote $L_{0}^{2}(\Omega):= H_{0}^{0}(\Omega)$.

For any given function $\phi(t)$ of $t$, we use $\phi^n$ to
denote an approximation of $\phi(n\tau)$, where $\tau$ is
the step-size. We will frequently use the shorthand
notations: $\delta_{t}\phi^{n+1}:=\phi^{n+1}-\phi^{n}$,
$\delta_{tt}\phi^{n+1}:=\phi^{n+1}-2\phi^{n}+\phi^{n-1}$,
$D_{\tau}\phi^{n+1}:=
\frac{3\phi^{n+1}-4\phi^{n}+\phi^{n-1}}{2\tau}
=\frac{1}{\tau}\delta_{t}\phi^{n+1}+\frac{1}{2\tau}\delta_{tt}\phi^{n+1}
$,
$\hat{\phi}^{n+\frac{1}{2}}
:=\frac{3}{2}\phi^{n}-\frac{1}{2}\phi^{n-1}$
and $\hat{\phi}^{n+1} :=2\phi^{n}-\phi^{n-1}$. Following
identities will be used frequently as well
\begin{align}\label{eq:ID:1}
2(h^{n+1}-h^n, h^{n+1}) =& \|h^{n+1}\|^2 - \|h^n\|^2 + \|h^{n+1}-h^n\|^2, \\
\label{eq:ID:2}
\begin{split}
(D_\tau h^{n+1}, h^{n+1}) =& \frac{1}{4\tau}(\|h^{n+1}\|^2 -\|h^{n}\|^2 \\ 
&\quad +\|2h^{n+1}-\!h^n\|^2
\!-\!\|2h^{n}\!-\!h^{n-1}\|^2
+\|\delta_{tt}h^{n+1}\|^2).
\end{split}
\end{align}
To prove energy stability of the numerical schemes, we
assume that the derivative of $f$ in equation \eqref{eq:AC}
is uniformly bounded, i.e.
\begin{equation}\label{eq:Lip}
\max_{\phi\in\mathbf{R}} | f'(\phi) | \le L,
\end{equation}
where $L$ is a non-negative constant. 

\begin{remark} \label{rmk:2.1} Note that the commonly
used double well potential does not satisfy the above
assumption. But, thanks to the maximum principle that the
Allen-Cahn equation
has~(cf. e.g. \cite{allen_microscopic_1979,
  chen_spectrum_1994, feng_numerical_2003,
  yang_error_2009}), the solution to equation \eqref{eq:AC}
is bounded by value $-1$ and $1$ if the initial
condition is bounded by $-1$ and $1$.  So it is safe to
modify the double-well energy $F(\phi)$ for $|\phi|$ larger
than $1$ to be quadratic growth without affecting the exact
solution if the initial condition is bounded by $-1$ and
$1$, such that assumption \eqref{eq:Lip} is satisfied.
This argument also applies to the assumption \eqref{eq:Lip2}
 in next section.
\end{remark}

\subsection{The stabilized linear BDF2 scheme.}\label{Subsec.1}
Suppose $\phi^0=\phi_0(\cdot)$ and
$\phi^1\approx \phi(\cdot,\tau)$ are given, our stabilized
linear BDF2 scheme (SL-BDF2) calculate
$\phi^{n+1}, n=1,2,\ldots,N=T/\tau-1$ iteratively, using
\begin{equation}\label{BDF1}
\frac{3\phi^{n+1}-4\phi^{n}+\phi^{n-1}}{2\tau\gamma }
=\varepsilon \Delta \phi^{n+1}
 -\frac{1}{\varepsilon}f(2\phi^{n}-\phi^{n-1})
  -A\tau \delta_{t}\phi^{n+1}
  -B\delta_{tt}\phi^{n+1},
\end{equation}
where $A$ and $B$ are two non-negative constants to
stabilize the scheme. 

\begin{thm}\label{thm1}
Assume that \eqref{eq:Lip} is satisfied. Under the condition
  \begin{equation}\label{eq:BDF:ABcond}
     A \geq \frac{L}{2 \varepsilon \tau}-\frac{1}{\tau^2 \gamma},
     \quad
    B\geq \frac{L}{ \varepsilon} - \frac{1}{2\tau\gamma},
  \end{equation}
 the following energy dissipation law
  \begin{multline}\label{eq:BDF:Edis}
      E_{B}^{n+1} \leq E_{B}^{n}
      -\frac{\varepsilon}{2}\|\nabla\delta_{t}\phi^{n+1}\|^{2}
      -\left(\frac{1}{\tau \gamma} +A\tau -\frac{L}{2 \varepsilon}\right)\|\delta_{t}\phi^{n+1}\|^{2}\\
      -\left(\frac{1}{4\tau\gamma} + \frac{B}{2} -\frac{L}{2 \varepsilon}\right)\|\delta_{tt}\phi^{n+1}\|^{2},
      \quad \forall\: n\geq1,
  \end{multline}
  holds for the scheme \eqref{BDF1}, where
  \begin{equation}\label{eq:BDF:E}
    E_{B}^{n+1}=E_{\varepsilon}(\phi^{n+1})
    +\left(\frac{1}{4\tau\gamma}+\frac{L}{2\varepsilon}+\frac{B}{2}\right)
    \|\delta_{t}\phi^{n+1}\|^{2}.
  \end{equation}
\end{thm}
\begin{proof}
 Pairing (\ref{BDF1}) with $\delta_t\phi^{n+1}$, we get
  \begin{equation}\label{cs7}
    \begin{split}
      \left(\frac{1}{\gamma}D_\tau\phi^{n+1},\delta_t\phi^{n+1}\right)
      = & \varepsilon (\Delta \phi^{n+1},\delta_t\phi^{n+1})
      -\frac{1}{\varepsilon}(f(\hat{\phi}^{n+1}),\delta_t\phi^{n+1})\\
      & - A\tau \| \delta_{t}\phi^{n+1} \|^2
      -B(\delta_{tt}\phi^{n+1},\delta_t\phi^{n+1}).
    \end{split}
  \end{equation}
  By integration by parts, following identities hold
  \begin{align}\label{cs8}
    \begin{split}
      \left(\frac1\gamma D_\tau\phi^{n+1},\delta_t\phi^{n+1}\right) 
      = &\frac{1}{\tau\gamma}\|\delta_{t}\phi^{n+1}\|^{2} \\
      &+\frac{1}{4\tau\gamma}\left(\|\delta_{t}\phi^{n+1}\|^{2}-\|\delta_{t}\phi^{n}\|^{2}
      +\|\delta_{tt}\phi^{n+1}\|^{2}\right),
   \end{split}\\
   \label{cs9}
      \varepsilon (\Delta \phi^{n+1},\delta_t\phi^{n+1})
      =& -\frac{\varepsilon}{2}(\|\nabla\phi^{n+1}\|^{2}-\|\nabla\phi^{n}\|^{2}
      +\|\nabla\delta_{t}\phi^{n+1}\|^{2}),\\    
    \label{cs11-1}
    -B(\delta_{tt}\phi^{n+1},\delta_t\phi^{n+1})
    =&-\frac{B}{2}\|\delta_{t}\phi^{n+1}\|^{2}+\frac{B}{2}\|\delta_{t}\phi^{n}\|^{2}
    -\frac{B}{2}\|\delta_{tt}\phi^{n+1}\|^{2}.
  \end{align}
  To handle the term involves $f$ in \eqref{cs7}, we expand
  $F(\phi^{n+1})$ and $F(\phi^n)$ at $\hat{\phi}^{n+1}$ as
  \begin{align*}\nonumber
    F(\phi^{n+1})&=F(\hat{\phi}^{n+1})+f(\hat{\phi}^{n+1})(\phi^{n+1}-\hat{\phi}^{n+1})+\frac{1}{2}f'(\zeta^{n}_{1})(\phi^{n+1}-\hat{\phi}^{n+1})^{2},\\
    F(\phi^{n})&=F(\hat{\phi}^{n+1})+f(\hat{\phi}^{n+1})(\phi^{n}-\hat{\phi}^{n+1})+\frac{1}{2}f'(\zeta^{n}_{2})(\phi^{n}-\hat{\phi}^{n+1})^{2},
  \end{align*}
  where $\zeta^{n}_1$ is a number between $\phi^{n+1}$ and
  $\hat{\phi}^{n+1}$, $\zeta^{n}_2$ is a number between
  $\phi^n$ and $\hat{\phi}^{n+1}$.  Taking the difference of
  above two equations, using the fact
  $\phi^{n+1}-\hat{\phi}^{n+1}=\delta_{tt}\phi^{n+1}$ and
  $\phi^{n}-\hat{\phi}^{n+1} = -\delta_t\phi^{n}$, we obtain
  \begin{equation}\label{bta4}
    \begin{split}
      F(\phi^{n+1})-F(\phi^{n})
      -f(\hat{\phi}^{n+1})\delta_t\phi^{n+1}
      ={}&
      \frac{1}{2}f'(\zeta^{n}_{1})(\delta_{tt}\phi^{n+1})^{2}
      -\frac{1}{2}f'(\zeta^{n}_{2})(\delta_{t}\phi^{n})^{2}\\
      \le{} &
      \frac{L}{2}|\delta_{tt}\phi^{n+1}|^2
      +\frac{L}{2}|\delta_{t}\phi^{n}|^2.
    \end{split}
  \end{equation}

  Taking inner product of the above equation with constant
  $1/\varepsilon$, then combining the result with
  (\ref{cs7}), (\ref{cs8}), (\ref{cs9}) and (\ref{cs11-1}), we obtain
  \begin{equation}\label{cs12}
    \begin{split}
      &\frac{1}{\varepsilon}(F(\phi^{n+1})-F(\phi^{n}),1)
      +\frac{\varepsilon}{2}(\|\nabla\phi^{n+1}\|^{2}-\|\nabla\phi^{n}\|^{2})\\
      &+\frac{1}{4\tau\gamma}(\|\delta_{t}\phi^{n+1}\|^{2}-\|\delta_{t}\phi^{n}\|^{2})
      +\Big(\frac{L}{2\varepsilon}+\frac{B}{2}\Big)(\|\delta_{t}\phi^{n+1}\|^{2}-\|\delta_{t}\phi^{n}\|^{2})\\
      &\leq
      -\frac{1}{4\tau\gamma}\|\delta_{tt}\phi^{n+1}\|^{2}
      -\frac{1}{\tau\gamma}\|\delta_{t}\phi^{n+1}\|^{2}
      -\frac{\varepsilon}{2}\|\nabla\delta_{t}\phi^{n+1}\|^{2}
      -A\tau\|\delta_{t}\phi^{n+1}\|^{2}\\
      &\quad+\frac{L}{2\varepsilon}\|\delta_{t}\phi^{n+1}\|^{2}
      -\frac{B}{2}\|\delta_{tt}\phi^{n+1}\|^{2}+\frac{L}{2\varepsilon}\|\delta_{tt}\phi^{n+1}\|^{2}.
    \end{split}
  \end{equation}
  Combining the above equation and the inequality
  $\frac{1}{\tau \gamma} + A\tau \geq \frac{L}{2
    \varepsilon},$
  $\frac{B}{2}+\frac{1}{4\tau\gamma} \geq \frac{L}{2 \varepsilon}$, we get  energy dissipation law
  \eqref{eq:BDF:Edis}.  \hfill
\end{proof}

\begin{remark}\label{bdf2}
  From equation \eqref{eq:BDF:ABcond}, we see that the
  SL-BDF2 scheme is stable with any non-negative $A$ including $A=0$, if
  \begin{equation}\label{eq:BDF:tcond0}
    \tau \leq \frac{2\varepsilon}{L \gamma},
  \end{equation}
  If one takes time step size even smaller, 
  \begin{equation}\label{eq:BDF:tcond1}
  \tau \leq \frac{\varepsilon}{2L \gamma},
  \end{equation}
  then the SL-BDF2 scheme is stable with any any
  non-negative $A$ and $B$, including the case $A=B=0$.
 
  On the other hand side, if we take 
  \begin{equation}\label{eq:ucondstabAB_BDF}
 A=\max_{\tau\ge 0}{\Big\{}\frac{L}{2 \varepsilon \tau}-\frac{1}{\tau^2 \gamma}{\Big\}}=\frac{\gamma L^2}{16\varepsilon^2},\quad B=\frac{L}{\varepsilon}
  \end{equation} 
  then the SL-BDF2 scheme is unconditional stable for any $\tau$.
\end{remark}

\subsection{The stabilized linear Crank-Nicolson scheme.}\label{Subsec.2}
Suppose $\phi^0=\phi_0(\cdot)$ and
$\phi^1\approx \phi(\cdot,\tau)$ are given, our stabilized
linear Crank-Nicolson scheme (SL-CN) calculate
$\phi^{n+1}, n=1,2,\ldots,N=T/\tau-1$ iteratively, using

\begin{equation}\label{accn1}
\frac{\phi^{n+1}-\phi^{n}}{\tau \gamma}=\varepsilon \Delta \left(\frac{\phi^{n+1}+\phi^{n}}{2} \right)
-\frac{1}{\varepsilon}f(\frac{3}{2}\phi^{n}- \frac{1}{2}\phi^{n-1})-A\tau \delta_{t}\phi^{n+1}
-B\delta_{tt}\phi^{n+1},
\end{equation}
where $A$ and $B$ are two non-negative constants.

\begin{thm}\label{accn}
Assume that \eqref{eq:Lip} is satisfied.  Under the condition
  \begin{equation}\label{accn2}
   A \geq \frac{L}{2 \varepsilon \tau}-\frac{1}{\tau^2 \gamma};
   \ \ B\geq \dfrac{L}{2 \varepsilon},
 \end{equation}
 the following energy law holds
 \begin{equation}\label{accn3}
   E_{C}^{n+1}\leq E_{C}^{n}-\left( 
   \frac{1}{\tau \gamma} +A\tau-\frac{L}{2\varepsilon}\right)\|\delta_{t}\phi^{n+1}\|^{2}
   -\left(\frac{B}{2}-\frac{L}{4\varepsilon}\right)\|\delta_{tt}\phi^{n+1}\|^{2},\quad \forall n\geq1,
 \end{equation}
 for the scheme (\ref{accn1}), where we define
 \begin{equation}\label{accn4}
   E_{C}^{n+1}=E(\phi^{n+1})
   +\left(\frac{L}{4\varepsilon} +\frac{B}{2}\right)\|\delta_{t}\phi^{n+1}\|^{2}.
 \end{equation}
\end{thm}

\begin{proof}
  Pairing the equation (\ref{accn1}) with
  $\delta_t\phi^{n+1}$, we get
  \begin{equation}\label{accn5}
    \begin{split}
      \frac{1}{\tau \gamma}\|\delta_{t}\phi^{n+1} \|^{2}=
      &-\frac{\varepsilon}{2}(\|\nabla \phi^{n+1}\|^{2} -
      \|\nabla \phi^{n}\|^{2})
      -\frac{1}{\varepsilon}\left(f(\frac{3}{2}\phi^{n}- \frac{1}{2}\phi^{n-1}),\phi^{n+1}-\phi^{n}\right)\\
      &-A\tau
      \|\delta_{t}\phi^{n+1}\|^{2}-\frac{B}{2}(\|\delta_{t}\phi^{n+1}\|^{2}-\|\delta_{t}\phi^{n}\|^{2}+\|\delta_{tt}\phi^{n+1}\|^{2}),
    \end{split}
  \end{equation}
  We use Taylor expansion at
  $\hat{\phi}^{n+\frac{1}{2}}=\frac{3}{2}\phi^{n}-\frac{1}{2}\phi^{n-1}$,
  \begin{align}\label{acta1}
    F(\phi^{n+1})&=F(\hat{\phi}^{n+\frac{1}{2}})+f(\hat{\phi}^{n+\frac{1}{2}})(\phi^{n+1}-\hat{\phi}^{n+\frac{1}{2}})
    +\frac{1}{2}f'(\eta^{n}_{1})(\phi^{n+1}-\hat{\phi}^{n+\frac{1}{2}})^{2},\\
    \label{acta2}
    F(\phi^{n})&=F(\hat{\phi}^{n+\frac{1}{2}})+f(\hat{\phi}^{n+\frac{1}{2}})(\phi^{n}-\hat{\phi}^{n+\frac{1}{2}})
    +\frac{1}{2}f'(\eta^{n}_{2})(\phi^{n}-\hat{\phi}^{n+\frac{1}{2}})^{2},
  \end{align}
  Subtracting (\ref{acta2}) from (\ref{acta1}) and the
  definition of $\hat{\phi}^{n+\frac{1}{2}}$, we have
  \begin{equation}\label{acta4}
    \begin{split}
      F(\phi^{n+1})-F(\phi^{n})
      =& f(\hat{\phi}^{n+\frac{1}{2}})(\phi^{n+1}-\phi^{n})+\frac{1}{2}f'(\eta^{n}_{1})(\phi^{n+1}-\hat{\phi}^{n+\frac{1}{2}})^{2}\\
      &\quad -\frac{1}{2}(f'(\eta^{n}_{1})+f'(\eta^{n}_{2})-f'(\eta^{n}_{1}))(\phi^{n}-\hat{\phi}^{n+\frac{1}{2}})^{2}\\
      =& f(\frac{3}{2}\phi^{n}-\frac{1}{2}\phi^{n-1})(\phi^{n+1}-\phi^{n})
      +\frac{1}{2}f'(\eta^{n}_{1})\delta_{t}\phi^{n+1}\delta_{tt}\phi^{n+1}\\
      &\quad
      -\frac{1}{8}(f'(\eta^{n}_{2})-f'(\eta^{n}_{1}))(\delta_{t}\phi^{n})^{2},
    \end{split}
  \end{equation}
  which give us
  \begin{equation}\label{accn6}
    \begin{split}
      &\frac{1}{\varepsilon}\left(f(\frac{3}{2}\phi^{n}
        -\frac{1}{2}\phi^{n-1}),\phi^{n+1}-\phi^{n} \right)\\
      &\quad=\frac{1}{\varepsilon}(F(\phi^{n+1})-F(\phi^{n}),1)
      -\frac{1}{2\varepsilon}(f'(\eta^{n}_{1}),\delta_{t}\phi^{n+1}\delta_{tt}\phi^{n+1})\\
      &\qquad+\frac{1}{8\varepsilon}(f'(\eta^{n}_{2})-f'(\eta^{n}_{1}),(\delta_{t}\phi^{n})^{2}).\\
    \end{split}
  \end{equation}
  Plugging (\ref{accn6}) into (\ref{accn5}), we obtain
  \begin{equation}\label{cn8}
    \begin{split}
      &\frac{\varepsilon}{2}(\|\nabla \phi^{n+1}\|^{2} -
      \|\nabla \phi^{n}\|^{2})
      +\frac{1}{\varepsilon}(F(\phi^{n+1})-F(\phi^{n}),1)
      +\frac{B}{2}(\|\delta_{t}\phi^{n+1}\|^{2}-\|\delta_{t}\phi^{n}\|^{2})\\
      &\leq -\frac{1}{\tau \gamma}\|\delta_{t}\phi^{n+1}
      \|^{2} -A\tau \|\delta_{t}\phi^{n+1}\|^{2}
      +\frac{L}{4\varepsilon}\|\delta_{t} \phi^{n+1}\|^{2}
      +\frac{L}{4\varepsilon}\|\delta_{t} \phi^{n}\|^{2}\\
      &\quad+\frac{L}{4\varepsilon}\|\delta_{tt} \phi^{n+1}\|^{2}
      -\frac{B}{2}\|\delta_{tt}\phi^{n+1}\|^{2}.
    \end{split}
  \end{equation}
  By the definition of $E_{C}^{n+1}$ and
  $A\tau+\frac{1}{\tau \gamma}\geq \frac{L}{2 \varepsilon}$,
  $\frac{B}{2} \geq \frac{L}{4 \varepsilon}$, we get 
  the desired results.  \hfill
\end{proof}
%

\begin{remark}\label{ac2}
	If we take
	 \begin{equation}\label{eq:ucondstabAB_CN}
	A=\frac{\gamma L^2}{16\varepsilon^2},\quad B=\frac{L}{2\varepsilon}
	\end{equation} 
	then the SL-CN scheme is unconditional stable for any $\tau$.
	
	On the other hand, by using the inequality
    $\| \delta_{tt} \phi^{n+1} \|^2 \leq 2\| \delta_t
    \phi^{n+1} \|^2 + 2\| \delta_t \phi^n \|^2$,
    it is easy to prove that when $A=B=0$, the SL-CN scheme
    (\ref{accn1}) is stable for
  \begin{equation}\label{eq:BDF:tcond}
    \tau \leq \dfrac{2\varepsilon}{3L\gamma }.
  \end{equation}
\end{remark}

\begin{remark}\label{rem:order}
	\yu{
	To make SL-BDF2 and SL-CN scheme be unconditionally stable, i.e. stable for 
	any time step size  $\tau>0$, we need take $A \sim 
	O({\gamma}/{\varepsilon^2})$. This seems that $A$ need to be very large in 
	a real simulation {since physically $\varepsilon$ is very small}. But 
	actually, it is not necessary. It is proved that the numerical interface 
	for the Allen-Cahn equatoin converges with the rate 
	$O(\varepsilon^2|\ln\varepsilon|^2)$ if no singularities 
	appear\cite{feng_numerical_2003,feng_analysis_2015}, which suggests that we 
	don't need to take $\varepsilon$ as small as the width of a physical 
	interface. Furthermore, $A$ has a linear dependence on the value of 
	$\gamma$. It was showed by Magaletti et al.~\cite{magaletti2013sharp} and 
	Xu et al.~\cite{xu_sharp-interface_2017} that the phase-field  
	Cahn-Hilliard--Navier-Stokes model for binary fluids has a fast convergence 
	with respect to $\varepsilon$ when the phenomenological mobility $\gamma 
	\sim O(\varepsilon^2)$. When coupled with hydrodynamics, {what is a 
	proper choice for the mobility $\gamma$ in the Allen-Cahn model is an 
	interesting question. We leave this to a future study.} 
}
\end{remark}

\begin{remark}\label{rem:li}
	\yu{
		Recently, Li, Qiao and Tang \cite{li_characterizing_2016}, 
		Li and Qiao \cite{li_second_2017} studied several first order and second
		order stabilized semi-implicit Fourier schemes, respectively, for the Cahn-Hilliard equation
		with double-well potential 
		\begin{equation} \label{eq:doublewell0}
		F(\phi)= \frac{1}{4}(\phi^2-1)^2.
		\end{equation}
		Without a Lipschitz condition on $F'(\phi)$, they proved that those schemes are unconditionally stable when very large stability constant $A$ used. 
		For example, according to Theorem 1.3 in \cite{li_second_2017}, for a 
		classical second order semi-implicit stabilized scheme
		proposed by Xu and Tang \cite{xu_stability_2006} applied to the 
		Cahn-Hilliard equation, the stabilization constant $A$ need to be as 
		large as $O(|\ln\varepsilon|^2/\varepsilon^8)$ to make the scheme 
		unconditionally stable (Note that the $A$ in \cite{li_second_2017} 
		corresponds to $\varepsilon B$ in this paper). However, the constants 
		$A, B$ in this paper are only of order $O(\gamma/\varepsilon^2), 
		O(1/\varepsilon)$, respectively.
		The reasons are in two aspects.
		Firstly, {the Cahn-Hilliard equation is much harder to solve than 
		the Allen-Cahn equation.} For the Allen-Cahn equation, since its 
		solution satisfies a maximum principle, 
		it is reasonable to modify $F$ defined in \eqref{eq:doublewell0} for
		$|\phi| > 1$, such that the Lipschitz condition \eqref{eq:Lip} is satisfied.
		Secondly, we use two stabilization terms instead of only one 
		stabilization term, {the extra one helps to maintain the stability 
		for larger time step sizes}.  The approach presented in this paper can 
		be extended to the 
		Cahn-Hilliard equation with {\em{quadratic growth}} energy as 
		well\cite{wang_two_2017, WangYu2017b}.
	}
\end{remark}

\section{Convergence analysis}

In this section, we shall establish the error estimate
of the two proposed schemes for the Allen-Cahn equation in the norm of
$l^{\infty}(0,T;L^{2})\cap l^{2}(0,T;H^{1})$.  
We will shown that, if the interface is well developed in the
initial condition, the error bounds depend on
$1/\varepsilon$ only in some lower polynomial order for
small $\varepsilon$.  Let $\phi(t^{n})$ be the exact
solution at time $t=t^{n}$ to the Allen-Cahn equation (\ref{eq:AC}) and
$\phi^{n}$ be the solution at time $t=t^{n}$ to the time
discrete numerical scheme \eqref{BDF1} (or \eqref{accn1}), we define error
function $e^{n}:=\phi^{n}-\phi(t^{n})$. Obviously $e^{0}=0$.

Before presenting the detailed error analysis, we first make
some assumptions. For simplicity, we take $\gamma=1$ in this
section, and assume $0<\varepsilon<1$.  We use notation
$\lesssim$ in the way that $f\lesssim g$ means that
$f \le C g$ with positive constant $C$ independent of
$\tau, \varepsilon$.   

\begin{assumption}\label{ap:1}
  We make following assumptions on $f$:
  $f=F'$, for $F\in C^{4}(\mathbf{R})$, such that
   $f'$ and $f''$ are uniformly bounded, i.e. $f$
    satisfies \eqref{eq:Lip} and
    \begin{equation}\label{eq:Lip2}
      \max_{\phi\in\mathbf{R}} | f''(\phi) | \le L_2,
    \end{equation}
    where $L_2$ is a non-negative constant.
\end{assumption}

Since the solution of Allen-Cahn equation satisfies maximum
principle (see Remark \ref{rmk:2.1}), one can always modify $f(\phi)$ for large
$|\phi|$ such that Assumption \ref{ap:1} hold without
affecting the exact solution.

\begin{assumption}\label{ap:02}
  \begin{itemize}
  \item[(i)] We assume that there exist non-negative
    constants $\sigma_{1}$ such that
    \begin{align}\label{ap:022}
      E_{\varepsilon}(\phi^{0}):=\frac{\varepsilon}{2}\|\nabla \phi^{0}\|^{2}+\frac{1}{\varepsilon}\|F(\phi^{0})\|_{L^{1}} &\lesssim \varepsilon^{-2\sigma_{1}},\\
      \label{init1}
      \|\phi_{t}^0\|^2&\lesssim \varepsilon^{-2\sigma_{1}-1},\\
      \label{init2}
      \|\nabla \phi_{t}^0\|^2&\lesssim \varepsilon^{-2\sigma_{1}-3},\\
      \label{init3}
      \|\nabla \phi_{tt}^0\|^2&\lesssim \varepsilon^{-2\sigma_{1}-7}.
   \end{align}

\item[(ii)] Assume that an appropriate scheme is used to
  calculate the numerical solution at first step, such that
\begin{align}\label{eq:AP:phi1}
  E_\varepsilon (\phi^1) \le E_\varepsilon (\phi^0) &\lesssim\varepsilon^{-2\sigma_{1}},\\
  \label{eq:AP:phi1mu2}
  \frac{1}{\tau }\|\delta_{t}\phi^{1}\|^{2} &\lesssim \varepsilon^{-2 \sigma_{1}}.
  \end{align}
Then it is easy to get
\begin{align}\label{ap:0220}
  E_{C}^{1} &\lesssim \varepsilon^{-2 \sigma_{1}},\\
  \label{ap:0221}
  E_{B}^{1}&\lesssim \varepsilon^{-2 \sigma_{1}}.
\end{align}

\item[(iii)] There exist a constant $\sigma_{0}>0$,
\begin{equation}\label{ap:022-2}
 \|e^{1}\|^{2}+\varepsilon \|\nabla e^{1}\|^{2} \lesssim \varepsilon^{-\sigma_{0}}\tau^{4}.
\end{equation}
\end{itemize}
\end{assumption}

Given Assumption \ref{ap:1} \ref{ap:02} (i), we have following
estimates for the exact solution to the Allen-Cahn equation.
\begin{lemma}\label{ap:03}
  Let $\phi$ be the exact solution of
  (\ref{eq:AC}), under the condition of Assumption
  \ref{ap:1} and \ref{ap:02} (i),
  the following regularities holds: \\
  
  \begin{enumerate}
  \item[(i)]
    $\int_{0}^{\infty}\|\phi_t\|^{2}{\rm d}t + E_{\varepsilon}(\phi) \lesssim \varepsilon^{-2\sigma_{1}}$;\\

  \item[(ii)]
    $ 2\varepsilon \int_{0}^{\infty}\|\nabla
    \phi_t\|^{2}{\rm d}t +{\rm ess} \sup
    \limits_{[0,\infty]} \|\phi_{t}\|^2
    \lesssim \varepsilon^{-2\sigma_{1}-1}$;\\

  \item[(iii)]
    $\int_{0}^{\infty}\|\phi_{tt}\|^{2}{\rm d}t + {\rm ess}
    \sup \limits_{[0,\infty]} \varepsilon\|\nabla
    \phi_{t}\|^2
    \lesssim \varepsilon^{-2\sigma_{1}-2}$;\\

  \item[(iv)]
    $\varepsilon \int_{0}^{\infty}\| \Delta
    \phi_{tt}\|^{2}{\rm d}t +{\rm ess} \sup
    \limits_{[0,\infty]} \|\nabla \phi_{tt} \|^2
    \lesssim \varepsilon^{-4\sigma_{1}-8} $;\\

  \item[(v)]
    $\int_{0}^{\infty}\|\phi_{ttt}\|^{2}{\rm d}t +{\rm ess}
    \sup \limits_{[0,\infty]} \varepsilon \|\nabla \phi_{tt}
    \|^2 \lesssim \varepsilon^{-4\sigma_{1}-7}$.
  \end{enumerate}
\end{lemma}

\begin{proof}   Take $\gamma=1$ in equation \eqref{eq:AC}, we have
	\begin{equation}\label{eq:AC11}
	\phi_{t}- \varepsilon \Delta \phi =- \dfrac{1}{\varepsilon}f(\phi). 
	\end{equation}
	
  \begin{itemize}
  \item[(i)]   Pairing \eqref{eq:AC11} with $\phi_t$ and taking integration
    by parts on the second term, we get
    \begin{equation}\label{eq:AC12}
      \|\phi_t\|^2 + \frac{\varepsilon}{2}\frac{d}{ d t} \|\nabla \phi\|^2
      =-\frac{1}{\varepsilon}(f(\phi),\phi_t)
      =-\frac{1}{\varepsilon}\frac{d}{d t}\int_{\Omega} |F(\phi)| dx.
    \end{equation}
    After integration over $[0, \infty]$ and using the inequality
    \eqref{ap:022}, we obtain (i).

  \item[(ii)] We differentiate \eqref{eq:AC11} in time to
    obtain
    \begin{equation}\label{eq:AC21}
      \phi_{tt}- \varepsilon \Delta \phi_{t} =- \dfrac{1}{\varepsilon}f(\phi)_t. \\
    \end{equation}
    Pairing \eqref{eq:AC21} with $\phi_t$ yields
    \begin{equation}\label{eq:AC22}
      \frac{1}{2}\frac{d}{d t}\|\phi_{t}\|^2+ \varepsilon \|\nabla \phi_{t}\|^2
      =- \dfrac{1}{\varepsilon}(f'(\phi)\phi_t, \phi_t)
      \leq \frac{1}{\varepsilon}\|f'(\phi)\|_{L^{\infty}}\|\phi_t\|^2. \\
    \end{equation}
    Integrate \eqref{eq:AC22} over $[0,\infty)$, yields
    \begin{equation}\label{eq:AC23}
      {\rm ess} \sup \limits_{[0,\infty]}  \|\phi_{t}\|^2+ 2\varepsilon\int_{0}^{\infty} \|\nabla \phi_{t}\|^2
      \lesssim \frac{2}{\varepsilon}\|f'(\phi)\|_{L^{\infty}} \int_{0}^{\infty } \|\phi_t\|^2 {\rm d} t
      +\|\phi_{t}^0\|^2. \\
    \end{equation}
    The assertion then follows from (i) and the inequality
    \eqref{init1}.  

  \item[(iii)] Testing \eqref{eq:AC21} with $\phi_{tt}$, we
    get
    \begin{equation}\label{eq:AC31}
      \begin{split}
        \|\phi_{tt}\|^2+ \frac{\varepsilon}{2}\frac{d}{dt}
        \|\nabla \phi_{t}\|^2
        =&- \dfrac{1}{\varepsilon}(f'(\phi)\phi_t, \phi_{tt})\\
        \leq&
        \frac{1}{2\varepsilon^2}\|f'(\phi)\|_{L^{\infty}}^2\|\phi_t\|^2
        +\frac{1}{2}\|\phi_{tt}\|^2. \\
      \end{split}
    \end{equation}
    Integrating \eqref{eq:AC31} over $[0,\infty)$, we get
    \begin{equation}\label{eq:AC32}
      \int_{0}^{\infty} \|\phi_{tt}\|^2 {\rm  d} t
      +  {\rm ess} \sup \limits_{[0,\infty]} \varepsilon\|\nabla \phi_{t}\|^2
      \lesssim \frac{1}{\varepsilon^2}\|f'(\phi)\|_{L^{\infty}}^2 \int_{0}^{\infty}\|\phi_t\|^2 {\rm d}t
      +\varepsilon\|\nabla \phi_{t}^0\|^2. \\
    \end{equation}
    and by using (i) and the inequality \eqref{init2} of
    Assumption \ref{ap:02}, we obtain (iii).   

  \item[(iv)] We differentiate \eqref{eq:AC21} in time to
    derive
    \begin{equation}\label{eq:AC41}
      \phi_{ttt}- \varepsilon \Delta \phi_{tt} =- \dfrac{1}{\varepsilon}f(\phi)_{tt}. \\
    \end{equation}
    Testing \eqref{eq:AC41} with $-\Delta \phi_{tt} $ and
    using $H^1(\Omega)\hookrightarrow L^4 (\Omega) $ for
    $d \leq 4$, we have
    \begin{equation}\label{eq:AC42}
      \begin{split}
        & \frac{1}{2} \frac{d}{ dt } \|\nabla \phi_{tt}\|^2 + \varepsilon \|\Delta \phi_{tt} \|^2\\
        =& \dfrac{1}{\varepsilon}(f(\phi)_{tt},\Delta \phi_{tt})\\
        =& \dfrac{1}{\varepsilon}(f''(\phi)\phi_{t}^2 + f'(\phi) \phi_{tt},\Delta \phi_{tt})\\
        \leq
        &\dfrac{1}{\varepsilon^3}(\|f''(\phi)\|_{L^{\infty}}^2
        \|\phi_t\|_{L^4}^{4}
        +\|f'(\phi)\|_{L^{\infty}}^2\|\phi_{tt}\|^2)
        +\frac{\varepsilon}{2}\|\Delta \phi_{tt}\|^2\\
        \leq &
        \dfrac{1}{\varepsilon^3}(C_{s}\|f''(\phi)\|_{L^{\infty}}^2
        (\|\nabla \phi_t\|^{4}+\|\phi_{t}\|^4)
        +\|f'(\phi)\|_{L^{\infty}}^2\|\phi_{tt}\|^2)
        +\frac{\varepsilon}{2}\|\Delta \phi_{tt}\|^2. \\
      \end{split}
    \end{equation}
    Integrating \eqref{eq:AC42} over $[0,\infty)$, we obtain
    \begin{equation}\label{eq:AC42-1}
      \begin{split}
        &{\rm ess} \sup \limits_{[0,\infty]} \|\nabla \phi_{tt}\|^2 + \varepsilon \int_{0}^{\infty}\|\Delta \phi_{tt} \|^2\\
        \lesssim &\dfrac{2}{\varepsilon^3}\left(
          C_{s}\|f''(\phi)\|_{L^{\infty}}^2
          \int_{0}^{\infty}( \|\nabla
          \phi_t\|^{4}+\|\phi_{t}\|^4) {\rm d}t
          +\|f'(\phi)\|_{L^{\infty}}^2  \int_{0}^{\infty}\|\phi_{tt}\|^2 {\rm d}t \right)\\
        &+ \|\nabla \phi_{tt}^0\|^2 \\
        \leq &\dfrac{2}{\varepsilon^3}
        C_{s}\|f''(\phi)\|_{L^{\infty}}^2 \left({\rm ess}
          \sup_{[0,\infty]} \|\nabla \phi_{t}\|^2
          \int_{0}^{\infty}\|\nabla \phi_t\|^{2} {\rm d}t
          +{\rm ess} \sup_{[0,\infty]} \| \phi_{t}\|^2 \int_{0}^{\infty}\|\phi_t\|^{2} {\rm d}t \right)\\
        &+\dfrac{2}{\varepsilon^3}\|f'(\phi)\|_{L^{\infty}}^2
        \int_{0}^{\infty}\|\phi_{tt}\|^2 {\rm d}t
        +\|\nabla \phi_{tt}^0\|^2. \\
      \end{split}
    \end{equation}
    The assertion then follows from (i) (ii) (iii) and the
    inequality \eqref{init3}.

  \item[(v)] Testing \eqref{eq:AC41} with $ \phi_{ttt} $, we
    have
    \begin{equation}\label{eq:AC51}
      \begin{split}
        \| \phi_{ttt}\|^2 + \frac{\varepsilon}{2}\frac{d
        }{dt} \|\nabla \phi_{tt} \|^2
        =& -\dfrac{1}{\varepsilon}(f(\phi)_{tt}, \phi_{ttt})\\
        \leq&
        \dfrac{1}{\varepsilon^2}(\|f''(\phi)\|_{L^{\infty}}^2
        \|\phi_t\|_{L^4}^{4}
        +\|f'(\phi)\|_{L^{\infty}}^2\|\phi_{tt}\|^2)
        +\frac{1}{2}\|\phi_{ttt}\|^2. \\
      \end{split}
    \end{equation}
    Integrating in time yields
    \begin{equation}\label{eq:AC52}
      \begin{split}
        &\int_{0}^{\infty}\| \phi_{ttt}\|^2 + {\rm ess} \sup \limits_{[0,\infty]} \varepsilon \|\nabla \phi_{tt} \|^2\\
        \lesssim &\dfrac{2}{\varepsilon^2}
        C_{s}\|f''(\phi)\|_{L^{\infty}}^2 \left({\rm ess}
          \sup_{[0,\infty]} \|\nabla \phi_{t}\|^2
          \int_{0}^{\infty}\|\nabla \phi_t\|^{2} {\rm d}t
          +{\rm ess} \sup_{[0,\infty]} \| \phi_{t}\|^2 \int_{0}^{\infty}\|\phi_t\|^{2} {\rm d}t \right)\\
        &+\dfrac{2}{\varepsilon^2}\|f'(\phi)\|_{L^{\infty}}^2
        \int_{0}^{\infty}\|\phi_{tt}\|^2 {\rm d}t
        +\varepsilon\|\nabla \phi_{tt}^0\|^2. \\
      \end{split}
    \end{equation}
    The assertion then follows from (i) (ii) (iii) and the
    inequality \eqref{init3}.
  \end{itemize}
  \hfill
\end{proof}

\subsection{Convergence analysis of the SL-BDF2 scheme}
Now, we present our first error estimate result, which is a
coarse estimate obtained by a standard approach.

\begin{prop}(Coarse error estimate)\label{prop1} Given
  Assumption \ref{ap:1} \ref{ap:02},
  $\forall \tau \leq\frac{1}{12}$, following error estimates
  hold for the SL-BDF2 scheme \eqref{BDF1}.
	\begin{equation}\label{BDF12-0}
	\begin{split}
	&\frac{1}{2}\|e^{n+1}\|^{2}+\|2e^{n+1}-e^{n}\|^{2}
	+2A\tau^2\| e^{n+1}\|^{2}
    +4\varepsilon\tau\|\nabla e^{n+1}\|^{2}\\
	 &+2A\tau^2\|\delta_{t} e^{n+1}\|^{2}
	 +\|\delta_{tt}e^{n+1}\|^{2}
	+4B\tau \|e^{n+1}\|^2\\
	\lesssim &
	\|e^{n}\|^{2}+\|2e^{n}-e^{n-1}\|^{2}
	+2A\tau^2\| e^{n}\|^{2}\\
	&+ \varepsilon^{- (4\sigma_1+7)} \tau^4
	+4\left(B^2+\frac{L^{2}}{\varepsilon^{2}}\right)\tau\|2e^n-e^{n-1}\|^2,
      \quad n\ge 1,
    \end{split}
  \end{equation}
  and
  \begin{equation}\label{BDF13-0}
	\begin{split}
      &\max_{1\leq n \leq N}
      \left(\|e^{n+1}\|^{2}+2\|2e^{n+1}-e^{n}\|^{2} +4
        A\tau^2\| e^{n+1}\|^{2}\right)
      +8  \varepsilon \tau \sum_{n=1}^{N}\|\nabla e^{n+1}\|^{2}\\
      &+ 4A\tau^2 \sum_{n=1}^{N}\|\delta_{t} e^{n+1}\|^{2}
      +2 \sum_{n=1}^{N}\|\delta_{tt}e^{n+1}\|^{2}
      +8  B\tau \sum_{n=1}^{N}\|e^{n+1}\|^2\\
      \lesssim{}& \exp \left(
        80\left(B^2+\frac{L^{2}}{\varepsilon^{2}}\right)T +
        12T \right) \varepsilon^{- \max\{4\sigma_1+7,
        \sigma_0\}} \tau^4.
	\end{split}
  \end{equation}
\end{prop}

\begin{proof}
  By taking the difference of equation \eqref{eq:AC} and
  \eqref{BDF1}, we obtain following error equation
  \begin{equation}\label{BDF2}
    \begin{split}
      D_{\tau}e^{n+1}=&\widetilde{R}_{1}^{n+1}+\varepsilon\Delta e^{n+1}-\frac{1}{\varepsilon}[f(2\phi^{n}-\phi^{n-1})-f(\phi(t^{n+1}))]\\
	  &-A\tau \delta_{t}e^{n+1} -B\delta_{tt}e^{n+1} -A
      \widetilde{R}_{2}^{n+1} -B \widetilde{R}_{3}^{n+1}.
    \end{split}
  \end{equation}
  where
  \begin{align*}
    D_{\tau}e^{n+1}:&=\frac{3e^{n+1}-4e^{n}+e^{n-1}}{2\tau},\\
    \widetilde{R}_{1}^{n+1}:&=\phi_{t}(t^{n+1})-D_{\tau}\phi(t^{n+1}),\\
    \widetilde{R}_{2}^{n+1}:&=\tau\delta_t\phi(t^{n+1})=\tau(\phi(t^{n+1})-\phi(t^{n})),\\
    \widetilde{R}_{3}^{n+1}:&=\delta_{tt}\phi(t^{n+1})=\phi(t^{n+1})-2\phi(t^{n})+\phi(t^{n-1}).\\
  \end{align*}
  Pairing (\ref{BDF2}) with $e^{n+1}$, we obtain
  \begin{equation}\label{BDF3}
    \begin{split}
      &(D_{\tau}e^{n+1},e^{n+1}) +\varepsilon\|\nabla
      e^{n+1}\|^{2}
	  +A\tau(\delta_{t} e^{n+1},e^{n+1})\\
	  ={}&(\widetilde{R}_{1}^{n+1}, e^{n+1}) -A(
      \widetilde{R}_{2}^{n+1}, e^{n+1})
      -B(\widetilde{R}_{3}^{n+1},e^{n+1})
	  -B(\delta_{tt}e^{n+1},e^{n+1})\\
	  &-\frac{1}{\varepsilon}\left(f(2\phi^{n}-\phi^{n-1})-f(\phi(t^{n+1})),e^{n+1}\right)\\
	  =&:J_{1}+J_{2}+J_{3}+J_{4}+J_{5}.
    \end{split}
  \end{equation}

  First, for the terms on the left side of (\ref{BDF3}), using 
  identity \eqref{eq:ID:2}, we
  have
  \begin{equation}\label{BDF4}
    \begin{split}
	  (D_{t}e^{n+1},e^{n+1})=
	  &\frac{1}{4\tau}(\|e^{n+1}\|^{2}+\|2e^{n+1}-e^{n}\|^{2})\\
	  &-\frac{1}{4\tau}(\|e^{n}\|^{2}+\|2e^{n}-e^{n-1}\|^{2})
	  +\frac{1}{4\tau}\|\delta_{tt}e^{n+1}\|^{2},
    \end{split}
  \end{equation}
  and using identity \eqref{eq:ID:1}, we get
  \begin{equation}\label{BDF5}
    A\tau(\delta_{t} e^{n+1}, e^{n+1})
    =\frac{1}{2}A\tau(\| e^{n+1}\|^{2}-\| e^{n}\|^{2}+\|\delta_{t} e^{n+1}\|^{2}).
  \end{equation}
  Then we estimate the terms on the right hand side of \eqref{BDF3}.
  \begin{align}\label{BDF6}
      J_{1}&=(\widetilde{R}_{1}^{n+1}, e^{n+1})
      \leq
      \|\widetilde{R}_{1}^{n+1}\|^{2}
      + \frac{1}{4}\| e^{n+1}\|^{2}, \\
  \label{BDF7}
      J_{2}&=-A( \widetilde{R}_{2}^{n+1}, e^{n+1})
      \leq A^{2}\|\widetilde{R}_{2}^{n+1}\|^{2}
      + \frac{1}{4}\| e^{n+1}\|^{2},\\
  \label{BDF8}
      J_{3}&=-B(\widetilde{R}_{3}^{n+1},e^{n+1})
      \leq
      B^{2}\|\widetilde{R}_{3}^{n+1}\|^{2}
      + \frac{1}{4}\|e^{n+1}\|^{2}\\
  \label{BDF9}
    \begin{split}
      J_{4}&=-B(\delta_{tt}e^{n+1},e^{n+1})
      =-B(e^{n+1}-(2e^n-e^{n-1}),e^{n+1})\\
      &\quad \leq -B\|e^{n+1}\|^2
      +B^2\|2e^n-e^{n-1}\|^2
      +\frac{1}{4}\| e^{n+1}\|^2.
    \end{split}
  \end{align}
  \begin{equation}\label{BDF10}
	\begin{split}
      J_{5}&=-\frac{1}{\varepsilon}\left(f(2\phi^{n}-\phi^{n-1})-f(\phi(t^{n+1})),e^{n+1}\right)\\
      & \leq \frac{L}{\varepsilon}\left(|2\phi^{n}-\phi^{n-1}-\phi(t^{n+1})|,|e^{n+1}|\right) \\
      & = \frac{L}{\varepsilon}\left(|2e^{n}-e^{n-1}-\delta_{tt}\phi(t^{n+1})|,|e^{n+1}|\right) \\
      &\leq \frac{L^{2}}{\varepsilon^{2}}\|2e^n-e^{n-1}\|^{2}
      +\frac{L^{2}}{\varepsilon^{2}}\|\widetilde{R}_{3}^{n+1}\|^{2}
      + \frac{1}{2}\|e^{n+1}\|^{2}.
	\end{split}
  \end{equation}

  Combining (\ref{BDF3})-(\ref{BDF10}) together,  yields
  \begin{equation}\label{BDF11}
	\begin{split}
      &\frac{1}{4\tau}(\|e^{n+1}\|^{2}+\|2e^{n+1}-e^{n}\|^{2})
      +\frac{1}{2}A\tau\| e^{n+1}\|^{2}\\
      &+\frac{1}{2}A\tau\|\delta_{t} e^{n+1}\|^{2}
      +\varepsilon\|\nabla e^{n+1}\|^{2}
      +\frac{1}{4\tau}\|\delta_{tt}e^{n+1}\|^{2}
      +B\|e^{n+1}\|^2\\
      \leq{}&
      \frac{1}{4\tau}(\|e^{n}\|^{2}+\|2e^{n}-e^{n-1}\|^{2})
      +\frac{1}{2}A\tau\| e^{n}\|^{2}\\
      &+\|\widetilde{R}_{1}^{n+1}\|^{2}
      +A^{2}\|\widetilde{R}_{2}^{n+1}\|^{2}
      +\left(B^2+\frac{L^{2}}{\varepsilon^{2}}\right)\|\widetilde{R}_{3}^{n+1}\|^{2}\\
      &+\left(B^2+\frac{L^{2}}{\varepsilon^{2}}\right)\|2e^n-e^{n-1}\|^2
      + \frac{3}{2}\| e^{n+1}\|^{2}.
	\end{split}
  \end{equation}
  By using Taylor expansions in integral form, one can get estimates
  for the residuals
  \begin{align} \label{eq:R1}
    \|\widetilde{R}_{1}^{n+1}\|^{2}&\leq
    8\tau^{3}\int_{t_{n-1}}^{t_{n+1}}\|\phi_{ttt}(t)\|^{2}{\rm
      d}t \lesssim \tau^3 \varepsilon^{-4\sigma_1-7},\\
 \label{eq:R2}
    \| \widetilde{R}_{2}^{n+1}\|^{2}&\leq
    \tau^{3}\int_{t_{n}}^{t_{n+1}}\|\phi_t(t)\|^{2}{\rm d}t
    \lesssim \tau^3 \varepsilon^{-2\sigma_1},\\
\label{eq:R3}
    \|\widetilde{R}_{3}^{n+1}\|^{2}&\leq
    6\tau^{3}\int_{t_{n-1}}^{t_{n+1}}\|\phi_{tt}(t)\|^{2}{\rm d}
    \lesssim \tau^3 \varepsilon^{-2\sigma_1-2}.
  \end{align}
  Taking $\tau \leq \frac{1}{12}$, combining \eqref{eq:R1}-\eqref{eq:R3}
  and the assumptions about the first step error, by using a discrete
  Gronwall inequality, we obtain \eqref{BDF13-0}. \eqref{BDF12-0} is obtained without using Gronwall inequality.
  \hfill
\end{proof}

Proposition \ref{prop1} is the usual error estimate, in which the error
growth depends on $1/\varepsilon$ exponentially.
To obtain a finer estimate on the error, we will need to use
a spectral estimate of the linearized Allen-Cahn operator
by Chen \cite{chen_spectrum_1994} for the case when the
interface is well developed in the Allen-Cahn system.

\begin{lemma}\label{lemma1} Let $\phi$ be the exact
  solution of Allen-Cahn equation \eqref{eq:AC} with
  interfaces are well developed
	in the initial condition (i.e. conditions (1.9)-(1.15) in
	\cite{chen_spectrum_1994} are satisfied).  
  Then there exist $0<\varepsilon_{0}\ll 1$ and positive constant $C_{0}$ such
  that the principle eigenvalue of the linearized Allen-Cahn operator
  $\mathcal{L}_{AC}:=-(\varepsilon\Delta-\frac{1}{\varepsilon}f'(\phi)I){/\varepsilon}$
  satisfies for all $t\in [0,T]$
  \begin{equation}\label{eigen}
    \lambda_{CH}=\inf_{\substack{0\neq v\in H^{1}(\Omega)\\ }}
    \frac{\varepsilon\|\nabla 
    v\|^{2}+\frac{1}{\varepsilon}(f'(\phi(\cdot,t))v,v)}{{\varepsilon}\|v\|^{2}}
    \geq-C_{0},
  \end{equation}
  for $\varepsilon\in (0,\varepsilon_{0})$. 
\end{lemma}

\begin{thm}\label{BDFerror}
  Suppose all of the Assumption \ref{ap:1},\ref{ap:02} hold. Let time
  step $\tau$ satisfy the following constraint
  \begin{equation}\label{finer-error-1}
    \tau \lesssim \min \left\{ \varepsilon^{2},\varepsilon^{\frac{1}{3}\max\{4\sigma_1+7, \sigma_{0}\}+\frac{5}{3}-\frac{2}{d}},
      \varepsilon^{\frac{1}{4}\max\{4\sigma_1+7, \sigma_{0}\}+ \frac{9}{2(6-d)}}\right\},
  \end{equation}
  then the solution of (\ref{BDF1})
  satisfies the following error estimate
  \begin{equation}\label{finer-error}
    \begin{split}
      &\max_{1 \leq n \leq N}\{ \|e^{n+1}\|^2 + 2\|2e^{n+1}-e^n\|^2 + 4 A \tau^2 \|e^{n+1}\|^2 \} \\
      &+ 4A \tau^2 \sum_{n=1}^{N} \|\delta_t e^{n+1}\|^2
      + \varepsilon^2 \sum_{n=1}^{N} \|\nabla e^{n+1}\|^2\\
      \leq & \exp{(8T(C_{0}{\varepsilon}+L+2))}\varepsilon^{- 
      \max\{4\sigma_1+7, \sigma_0\}} \tau^4.
    \end{split}
  \end{equation}
\end{thm}

\begin{proof}
  We refine the result of Proposition \ref{prop1} by
  re-estimating $J_4$ in equation \eqref{BDF3} as
  \begin{equation}\label{J4}
	\begin{split}
      J_{4}&=-B(\delta_{tt}e^{n+1},e^{n+1})
      \leq
      B^2\|\delta_{tt}e^{n+1}\|^2
      +\frac{1}{4}\| e^{n+1}\|^2,
	\end{split}
  \end{equation}
  and rewriting $J_5$ as
  \begin{align}\label{J5}
	J_5 &=  J_6 + J_7,\\
	\begin{split}\label{J6}
      J_{6}&=-\frac{1}{\varepsilon}\left(f(2\phi^{n}-\phi^{n-1})-f(\phi^{n+1}),e^{n+1}\right)\\
      &\quad\le \frac{L}{\varepsilon}(|\delta_{tt}e^{n+1}|+|\widetilde{R}_{3}^{n+1}|, |e^{n+1}|)\\
      &\quad\le \frac{L^2}{\varepsilon^2}\left(\|\delta_{tt}e^{n+1}\|^2
        +\|\widetilde{R}_{3}^{n+1}\|^2\right) + \frac{1}{2} \| e^{n+1}\|^2,
	\end{split}\\
	\begin{split}\label{J7-0}
      J_{7}&=
      -\frac{1}{\varepsilon}\left(f(\phi^{n+1})-f(\phi(t^{n+1})),e^{n+1}\right)\\
      &\quad\le -\frac{1}{\varepsilon}\left(f'(\phi(t^{n+1}))e^{n+1},e^{n+1}\right)
      +\frac{L_2}{\varepsilon}\|e^{n+1}\|_{L^3}^3.
	\end{split}
  \end{align}
  The spectrum estimate \eqref{eigen} give us
  \begin{equation}\label{spectrumBDF}
    \varepsilon \|\nabla e^{n+1}\|^{2}
    +\frac{1}{\varepsilon}(f'(\phi(t^{n+1}))e^{n+1},  e^{n+1})
    \geq -{\varepsilon}C_{0}\|e^{n+1}\|^{2}.
  \end{equation}
  Applying (\ref{spectrumBDF}) with a scaling factor $-(1-\varepsilon)$, we get
  \begin{equation}\label{spectrum1BDF}
    -(1-\varepsilon)\frac{1}{\varepsilon}(f'(\phi(t^{n+1}))e^{n+1},  e^{n+1})
    \leq C_{0}{\varepsilon}(1-\varepsilon)\|e^{n+1}\|^{2} + 
    (1-\varepsilon)\varepsilon \|\nabla e^{n+1}\|^{2}.
  \end{equation}
  On the other hand,
  \begin{equation}\label{BDF12}
    -(f'(\phi(t^{n+1}))e^{n+1},  e^{n+1})
    \leq L\|e^{n+1}\|^{2}.
  \end{equation}
  Now, we estimate the $L^{3}$ term in \eqref{J7-0} by interpolating $L^{3}$ between $L^{2}$ and $H^{1}$
  \[\|e^{n+1}\|_{L^{3}}^{3}\leq K (
  \|\nabla e^{n+1}\|^{\frac{d}{2}}
  \|e^{n+1}\|^{\frac{6-d}{2}}+
  \|e^{n+1}\|^{3}),\]
  where K is a constant independent of $\varepsilon$ and $\tau$.
  We continue the estimate by using Young's Gronwall inequality
  \begin{equation}\label{BDF13}
    \frac{L_2 }{\varepsilon}\|e^{n+1}\|^{3}_{L^{3}}
    \leq \frac{d}{6} \varepsilon^{\frac{3}{d}} \|\nabla e^{n+1}\|^{3}
    +\frac{6-d}{6}\Big{(} \frac{L_{2}K}{\varepsilon^{\frac{3}{2}}}\Big{)}^{\frac{6}{6-d}}
    \|e^{n+1}\|^{3} + \frac{L_2K}{\varepsilon}\|e^{n+1}\|^{3}.
  \end{equation}
  Substituting (\ref{spectrum1BDF}) (\ref{BDF12}) (\ref{BDF13}) into (\ref{J7-0}), we get
  \begin{multline}\label{J7}
      J_7\leq (C_{0}{\varepsilon}(1-\varepsilon)+L)\|e^{n+1}\|^{2} + 
      (1-\varepsilon)\varepsilon\|\nabla e^{n+1}\|^{2} +
      \frac{d}{6} \varepsilon^{\frac{3}{d}} \|\nabla e^{n+1}\|^{3}\\
      +\Big{(}\frac{6-d}{6}\Big{(} \frac{L_{2}K}{\varepsilon^{\frac{3}{2}}}\Big{)}^{\frac{6}{6-d}}
      + \frac{L_2K}{\varepsilon}\Big{)}\|e^{n+1}\|^{3}.
  \end{multline}

  Substituting the estimate of \eqref{BDF4}-\eqref{BDF8},
  \eqref{J4}-\eqref{J6} and
  \eqref{J7} into (\ref{BDF3}), we get
  \begin{equation}\label{BDF14}
    \begin{split}
	  &\frac{1}{4\tau}((\|e^{n+1}\|^{2}+\|2e^{n+1}-e^{n}\|^{2})
	  -(\|e^{n}\|^{2}+\|2e^{n}-e^{n-1}\|^{2}))\\
      & +\frac{1}{2}A\tau(\| e^{n+1}\|^{2}-\| e^{n}\|^{2})
      +\frac{1}{2}A\tau\|\delta_{t} e^{n+1}\|^{2}+
      \frac{1}{4\tau}\|\delta_{tt}e^{n+1}\|^{2} + \varepsilon^2 \|\nabla e^{n+1}\|^2\\
      \leq &\|\widetilde{R}^{n+1}_1\|^{2} + A^2\|\widetilde{R}_{2}^{n+1}\|^2+
      \Big{(} B^2 + \frac{L^2}{\varepsilon^2}\Big{)}\|\widetilde{R}_3^{n+1}\|^2
      \\
      &+\left( C_{0}{\varepsilon}(1-\varepsilon)+L + \frac{3}{2} 
      +G^{n+1}\right) \|e^{n+1}\|^{2}
      + \Big{(} B^2 + \frac{L^2}{\varepsilon^2}\Big{)}\|\delta_{tt}e^{n+1}\|^2\\
      &+
      Q^{n+1} \|\nabla e^{n+1}\|^{2},
    \end{split}
  \end{equation}
  where
  $Q^{n+1}=\frac{d}{6} \varepsilon^{\frac{3}{d}} \|\nabla
  e^{n+1}\|$,
  $G^{n+1}=\Big{(}\frac{6-d}{6}\Big{(}
  \frac{L_{2}K}{\varepsilon^{\frac{3}{2}}}\Big{)}^{\frac{6}{6-d}}+\frac{L_2K}{\varepsilon}\Big{)}
  \|e^{n+1}\|$.\\
  If $Q^{n+1}$ is uniformly bounded by constant
  $\frac{\varepsilon^{2}}{2}$, $G^{n+1}$ is uniformly
  bounded by constant $\frac{1}{2}$, then choose
  $\tau \leq \max \{\frac{\varepsilon^2}{4(B^2 \varepsilon^2
    + L^2)}, \frac{1}{8 \left( C_0 {\varepsilon}\left( 1-
        \varepsilon\right) +L+2 \right)} \} $,
  by Gronwall inequality and the first step error estimate
  \eqref{ap:022-2} in Assumption \ref{ap:02},
  we will get the finer error estimate (\ref{finer-error}).

  We prove this by induction. Assuming that
  the finer estimate holds for all first ${n\le} N$ time steps:
  \begin{equation}\label{finer-errorN}
    \begin{split}
      &\max_{1 \leq n \leq N}\{ \|e^{n}\|^2 + 2\|2e^{n}-e^{n-1}\|^2 + 4 A \tau^2 \|e^{n}\|^2 \} \\
      &+ 4A \tau^2 \sum_{n=1}^{N} \|\delta_t e^{n}\|^2
      + 4\tau \varepsilon^2 \sum_{n=1}^{N} \|\nabla e^{n}\|^2\\
      \leq & \exp{(8T(C_{0}{\varepsilon}+L+2))}\varepsilon^{- 
      \max\{4\sigma_1+7, \sigma_0\}} \tau^4.
    \end{split}
  \end{equation}
  Combining \eqref{finer-errorN} with the coarse estimate \eqref{BDF12-0} leads to
  \begin{equation}\label{BDF15}
	\begin{split}
      &\|e^{N+1}\|^{2}+2\|2e^{N+1}-e^{N}\|^{2}
      +4A\tau^2\| e^{N+1}\|^{2}\\
      &+4A\tau^2\|\delta_{t} e^{N+1}\|^{2}
      +8\varepsilon\tau\|\nabla e^{N+1}\|^{2}
      +2\|\delta_{tt}e^{N+1}\|^{2}
      +8B\tau \|e^{N+1}\|^2\\
      \lesssim &
      \varepsilon^{- \max\{4\sigma_1+7, \sigma_0\}} \tau^4,
      \quad N\ge 1.
    \end{split}
  \end{equation}
  Then by taking $\tau \lesssim \varepsilon^{\frac{1}{3} \max\{4\sigma_1+7, \sigma_0\}+\frac{5}{3} -\frac{2}{d}}$, we have
  \begin{equation}\label{t9}
    Q^{N+1} \lesssim
    \varepsilon^{\frac{3}{d}}\varepsilon^{-\frac{1}{2} \max\{4\sigma_1+7, \sigma_0\}-\frac{1}{2}}\tau^{\frac{3}{2}}
    \lesssim \frac{\varepsilon^{2}}{2},
  \end{equation}
  By taking $\tau \lesssim\varepsilon^{\frac{1}{4} \max\{4\sigma_1+7, \sigma_0\}+ \frac{9}{2(6-d)}}$, we have
  \begin{equation}\label{t10}
    G^{N+1}
    \lesssim  \varepsilon^{-\frac{9}{6-d}}\varepsilon^{-\frac{1}{2}\max\{4\sigma_1+7, \sigma_{0}\}}\tau^{2}
    \lesssim \frac{1}{2}.
  \end{equation}
  So, by taking step-sizes as defined in \eqref{finer-error-1},
  the finer error estimate for $N+1$ step can be obtained, and the the proof is completed by mathematical induction.
  \hfill
\end{proof}


\subsection{Convergence analysis of the SL-CN scheme} 

Similar as the error estimate of SL-BDF2 scheme, we first
present the coarse error estimate for SL-CN scheme.

\begin{prop}(Coarse error estimate)\label{prop2} Given
  Assumption \ref{ap:1} \ref{ap:02},
  $\forall \tau \lesssim \varepsilon$, following error
  estimate holds for the SL-CN scheme \eqref{accn1}.
  \begin{equation}\label{cet2-10}
    \begin{split}
      &\frac{1}{2}\|e^{n+1}\|^{2} +2\tau \varepsilon\|\nabla
      \frac{e^{n+1}+e^{n}}{2}\|^{2}
      +A\tau^{2}\|e^{n+1}\|^{2}
      +B\tau \|e^{n+1}\|^{2}\\
      \lesssim&\varepsilon^{-\max\{4\sigma_1+7,
        \sigma_{0}\}}\tau^{4} +2B\tau \|e^{n}-e^{n-1}\|^{2}
      +\frac{2L}{\varepsilon}\tau \| \frac{3}{2}e^{n}-\frac{1}{2}e^{n-1} \|^{2}\\
      &
      +\left(\frac{5}{2}+\frac{B}{2}+\frac{L}{2\varepsilon}\right)\tau
      \|e^{n}\|^{2} + 2A\tau^{2}\| e^{n}\|^{2} +2B\tau
      \|e^{n}\|^{2}, \quad\forall n\geq 1.
    \end{split}
  \end{equation}
  and
  \begin{equation}\label{cet5-1}
    \begin{split}
      &\max_{1\leq n\leq N}\left(\|e^{n+1}\|^{2}
        +2A\tau^{2}\| e^{n+1}\|^{2}
        +2B\tau\|e^{N+1}\|^{2}\right)
      +4\varepsilon \tau\sum_{n=1}^{N}\|\nabla \frac{e^{n+1}+e^{n}}{2}\|^{2}\\
      \lesssim & \mathrm{exp }\Big{(}17B+
      5+\frac{11L}{\varepsilon}\Big{)}T
      \varepsilon^{-\max\{4\sigma_1+7,
        \sigma_{0}\}}\tau^{4}.
    \end{split}
  \end{equation}
\end{prop}

\begin{proof}
  The following equation for the error functions holds:

  \begin{equation}\label{cet}
    \begin{split}
      \frac{e^{n+1}- e^{n}}{\tau} =&R_1^{n+1}
      +\varepsilon\Delta \frac{e^{n+1}+e^{n}}{2} -\frac{1}{\varepsilon}\Big{(}f(\frac{3}{2}\phi^{n}-\frac{1}{2}\phi^{n-1})-f(\phi(t^{n+\frac{1}{2}}))\Big{)}\\
      &-A\tau \delta_{t}e^{n+1}-A R_2^{n+1}
      -B\delta_{tt}e^{n+1} -BR_3^{n+1} +\varepsilon \Delta
      R_4^{n+1}.
    \end{split}
  \end{equation}
  where
  \begin{align}\label{e3}
    R_{1}^{n+1}&=\phi_{t}^{n+\frac{1}{2}}-\frac{\phi(t^{n+1})-\phi(t^{n}))}{\tau},\\
  \label{e5}
    R_{2}^{n+1}&=\tau(\phi(t^{n+1})-\phi(t^{n})),\\
  \label{e4} 
  R_{3}^{n+1}&=
    \phi(t^{n+1})-2\phi(t^{n})+\phi(t^{n-1}),\\
  \label{e4-4} 
  R_{4}^{n+1}&=
    \frac{\phi(t^{n+1})+\phi(t^{n})}{2}-\phi(t^{n+\frac{1}{2}}).
  \end{align}

  Pairing (\ref{cet}) with
  $ \frac{e^{n+1}+e^{n}}{2}$, we get
  \begin{equation}\label{cet2}
    \begin{split}
      &\frac{1}{2\tau }(\|e^{n+1}\|^{2}-\|e^{n}\|^{2})
      +\varepsilon\|\nabla \frac{e^{n+1}+e^{n}}{2}\|^{2}
      +\frac{A\tau}{2}(\| e^{n+1}\|^{2}-\| e^{n}\|^{2})\\
      =&\Big{(} R_{1}^{n+1},\frac{e^{n+1}+e^{n}}{2}\Big{)}
      -A\Big{(} R_{2}^{n+1},\frac{e^{n+1}+e^{n}}{2}\Big{)}
      -B \Big{(}R_{3}^{n+1},\frac{e^{n+1}+e^{n}}{2}\Big{)}\\
      &+\varepsilon \Big{(} \Delta
      R_{4}^{n+1},\frac{e^{n+1}+e^{n}}{2}\Big{)}
      -B\Big{(}\delta_{tt}e^{n+1},\frac{e^{n+1}+e^{n}}{2}\Big{)}\\
      &-\frac{1}{\varepsilon}\Big{(}f(\frac{3}{2}\phi^{n}-\frac{1}{2}\phi^{n-1})
      -f(\phi(t^{n+\frac{1}{2}})),\frac{e^{n+1}+e^{n}}{2}\Big{)}\\
      =&:J_{1}+J_{2}+J_{3}+J_{4}+J_{5}+J_{6}=:J,\\
    \end{split}
  \end{equation}

  For the right hand of (\ref{cet2}), by using
  Cauchy-Schwarz inequality, we obtain the following
  estimate:
  \begin{alignat}{2}
  \label{ej1}
    J_{1}&=\Big{(}R_{1}^{n+1},\frac{e^{n+1}+e^{n}}{2}\Big{)}
    &\leq \|R_{1}^{n+1}\|^{2}+\frac{1}{4}\| \frac{e^{n+1}+e^{n}}{2}\|^{2},\\
  \label{ej2}
    J_{2}&=A\Big{(} R_{2}^{n+1},\frac{e^{n+1}+e^{n}}{2}\Big{)}
    &\leq A^{2}\| R_{2}^{n+1}\|^{2}+\frac{1}{4}\| \frac{e^{n+1}+e^{n}}{2}\|^{2},\\
  \label{ej3}
    J_{3}&=-B \Big{(}R_{3}^{n+1},\frac{e^{n+1}+e^{n}}{2}\Big{)}
    &\leq B^{2}\|R_{3}^{n+1}\|^{2}+\frac{1}{4}\| \frac{e^{n+1}+e^{n}}{2}\|^{2},\\
  \label{ej4}
    J_{4}&=\varepsilon \Big{(} \Delta R_{4}^{n+1},\frac{e^{n+1}+e^{n}}{2}\Big{)}
    &\leq \varepsilon^{2}\|\Delta R_{4}^{n+1}\|^{2}+\frac{1}{4}\| \frac{e^{n+1}+e^{n}}{2}\|^{2},
  \end{alignat}
  For $J_{5}$ of the right side of (\ref{cet2}), by using
  the equation
  $\delta_{tt}e^{n+1}=\delta_{t}e^{n+1}-\delta_{t}e^{n}$, we
  have
  \begin{equation}\label{e11}
    \begin{split}
      J_{5}&=-B \Big{(}\delta_{tt}e^{n+1},\frac{e^{n+1}+e^{n}}{2}\Big{)}\\
      &=-\frac{B}{2}(\|e^{n+1}\|^{2}-\|e^{n}\|^{2})+\frac{B}{2}(\delta_{t}e^{n},e^{n+1}+e^{n})\\
      &\leq-\frac{B}{2}(\|e^{n+1}\|^{2}-\|e^{n}\|^{2})
      +B\|e^{n}-e^{n-1}\|^{2} + \frac{B}{4}\|
      \frac{e^{n+1}+e^{n}}{2}\|^{2}.
    \end{split}
  \end{equation}

  \begin{equation}\label{e12}
    \begin{split}
      J_{6}&=-\frac{1}{\varepsilon}\Big{(}f(\frac{3}{2}\phi^{n}-\frac{1}{2}\phi^{n-1})-f(\phi(t^{n+\frac{1}{2}})),\frac{e^{n+1}+e^{n}}{2}\Big{)}\\
      &\leq \frac{L}{\varepsilon}\Big{(}| R_{5}^{n+1}|,|\frac{e^{n+1}+e^{n}}{2}|\Big{)}+ \frac{L}{\varepsilon}\Big{(} |\frac{3}{2}e^{n}-\frac{1}{2}e^{n-1}|,|\frac{e^{n+1}+e^{n}}{2}|\Big{)}\\
      &\leq \frac{L^{2}}{\varepsilon^{2}}\|R_{5}^{n+1}\|^{2}
      +\frac{1}{4}\| \frac{e^{n+1}+e^{n}}{2}\|^{2}
      +\frac{L}{\varepsilon}\|
      \frac{3}{2}e^{n}-\frac{1}{2}e^{n-1} \|^{2}
      +\frac{L}{4\varepsilon}\|
      \frac{e^{n+1}+e^{n}}{2}\|^{2},
    \end{split}
  \end{equation}
  where
  \begin{equation} \label{e5-5} {R_{5}^{n+1}=}
    \frac{3}{2}\phi(t^{n})-\frac{1}{2}\phi(t^{n-1})-\phi(t^{n+\frac{1}{2}}).
  \end{equation}

  Substituting $J_{1},\cdots, J_{6}$ into (\ref{cet2}), we
  have
  \begin{equation}\label{cet2-1}
    \begin{split}
      &\frac{1}{2\tau }(\|e^{n+1}\|^{2}-\|e^{n}\|^{2})
      +\varepsilon\|\nabla \frac{e^{n+1}+e^{n}}{2}\|^{2}
      +\frac{A\tau}{2}(\|e^{n+1}\|^{2}-\| e^{n}\|^{2})\\
      &+\frac{B}{2}(\|e^{n+1}\|^{2}-\|e^{n}\|^{2})\\
      \leq &\left(\|R_{1}^{n+1}\|^{2} + A^{2}\|
        R_{2}^{n+1}\|^{2}
        +B^{2}\|R_{3}^{n+1}\|^{2}+\varepsilon^{2}\|\Delta
        R_{4}^{n+1}\|^{2}
        +\frac{L^{2}}{\varepsilon^{2}}\|R_{5}^{n+1}\|^{2} \right)\\
      &+B\|e^{n}-e^{n-1}\|^{2} +\frac{L}{\varepsilon}\|
      \frac{3}{2}e^{n}-\frac{1}{2}e^{n-1} \|^{2}
      +\Big{(}\frac{5}{4}+\frac{B}{4}+\frac{L}{4\varepsilon}\Big{)}\|\frac{e^{n+1}+e^{n}}{2}\|^{2},
    \end{split}
  \end{equation}
  By using Taylor expansions in integral form, one can get
  estimates for the residuals
  \begin{align} 
  	\label{R1} 
  	\|R_{1}^{n+1}\|^{2}\leq&
    \tau^{3}\int_{t^{n}}^{t^{n+1}}\|\phi_{tt}(t)\|^{2}{\rm
      d}t \lesssim \varepsilon^{-2\sigma_1-2} \tau^3,\\
  \label{R2} 
  \| R_{2}^{n+1}\|^{2}\leq&
    \tau^{3}\int_{t^{n}}^{t^{n+1}}\| \phi_t(t)\|^{2}{\rm d}t
    \lesssim \varepsilon^{-2\sigma_1} \tau^3,\\
  \label{R3}
    \|R_{3}^{n+1}\|^{2}\leq& 6\tau^{3}\int_{t^{n-1}}^{t^{n+1}}\|\phi_{tt}(t)\|^{2}{\rm d}t
    \lesssim  \varepsilon^{-2\sigma_1-2} \tau^3,\\
  \label{R4}
    \|\Delta R_{4}^{n+1}\|^{2}\leq&
    \tau^{3}\int_{t^{n}}^{t^{n+1}}\| \Delta \phi_{tt}(t)\|^{2}{\rm d}t
    \lesssim \varepsilon^{-4\sigma_1-9} \tau^3,\\
  \label{R5}
    \|R_{5}^{n+1}\|^{2}\leq&
    \tau^{3}\int_{t^{n-1}}^{t^{n+1}}\|\phi_{tt}(t)\|^{2}{\rm d}t
    \lesssim \varepsilon^{-2\sigma_1-2} \tau^3.
  \end{align}
  Taking
  $\tau<1/\big{(}\frac{5}{2}+\frac{B}{2}+\frac{L}{2\varepsilon}\big{)}\lesssim\varepsilon$,
  combining \eqref{R1}-\eqref{R5} and the error assumption
  of the first step, by using a discrete Gronwall inequality,
  one get (\ref{cet5-1}). \eqref{cet2-10} is obtained without
  using Gronwall inequality.  \hfill
\end{proof}

Proposition \ref{prop2} is the usual error estimate, in
which the error growth depends on $1/\varepsilon$
exponentially.  Next, we give a finer error estimate by
using Lemma \ref{lemma1}.

\begin{thm}\label{CNerror}
  Suppose all of the Assumption \ref{ap:1},\ref{ap:02}
  hold. Let $\tau$ satisfy the following constraint
  \begin{equation}\label{error-1}
    \tau \lesssim \min \left\{ \varepsilon^{2},\varepsilon^{\frac{1}{3}\max\{4\sigma_1+11, \sigma_{0}\}+\frac{5}{3}-\frac{2}{d}},
      \varepsilon^{\frac{1}{4}\max\{4\sigma_1+11, \sigma_{0}\}+ \frac{9}{2(6-d)}}\right\},
 \end{equation}
 then the solution of (\ref{accn1}) satisfies the following
 error estimate
 \begin{equation}\label{error-2}
   \begin{split}
     &\max_{1\leq n \leq
       N}\left\{\|e^{n+1}\|^{2}+2\varepsilon \tau \|\nabla
       e^{n+1}\|^{2}
       +\left( 2B+8B^{2}+\frac{2L}{\varepsilon}+ \frac{2L^{2}}{ \varepsilon^{2}}\right)\tau \|\delta_{t}e^{n+1}\|^{2} \right\}\\
     & +2\varepsilon^2 \tau\sum_{n=1}^{N}\|\nabla \frac{e^{n+1}+e^{n}}{2}\|^{2}\\
     \lesssim & \exp ((12+ 4C_{0}{\varepsilon} +4L)T)
     \varepsilon^{-\max\{4\sigma_1+11, \sigma_{0}\}}
     \tau^{4}.\\
   \end{split}
 \end{equation}
\end{thm}

\begin{proof}
  To get a better convergence results, we re-estimate $J_5$ in \eqref{e11}
  as
  \begin{equation}\label{e111}
    J_{5}= -B \left(\delta_{tt}e^{n+1},\frac{e^{n+1}+e^{n}}{2}\right)
    \leq B^{2}\|\delta_{tt}e^{n+1}\|^{2}+\frac{1}{4}\| \frac{e^{n+1}+e^{n}}{2}\|^{2}.
  \end{equation}

  For $J_{6}$, we have
  \begin{align}\label{j1}
    J_{6}=&-\frac{1}{\varepsilon}\Big{(}f(\frac{3}{2}\phi^{n}-\frac{1}{2}\phi^{n-1})-f(\phi(t^{n+\frac{1}{2}})),\frac{e^{n+1}+e^{n}}{2}\Big{)}\\
    =&-\frac{1}{\varepsilon}\Big{(}f(\frac{3}{2}\phi^{n}-\frac{1}{2}\phi^{n-1})-f(\frac{\phi^{n+1}+\phi^{n}}{2}),\frac{e^{n+1}+e^{n}}{2}\Big{)}\\
          &-\frac{1}{\varepsilon}\Big{(}f(\frac{\phi^{n+1}+\phi^{n}}{2})-f(\phi(t^{n+\frac{1}{2}})),\frac{e^{n+1}+e^{n}}{2}\Big{)}\\
    :=&J_{7}+J_{8}.
  \end{align}

  \begin{align}\label{j2}
    J_{7}
    =&-\frac{1}{\varepsilon}\Big{(}f(\frac{3}{2}\phi^{n}-\frac{1}{2}\phi^{n-1})-f(\frac{\phi^{n+1}+\phi^{n}}{2}),\frac{e^{n+1}+e^{n}}{2}\Big{)}\\
    \leq & \frac{L}{2 \varepsilon}\Big{(}| \delta_{tt}\phi^{n+1}|,|\frac{e^{n+1}+e^{n}}{2}|\Big{)}\\
    =&  \frac{L}{2 \varepsilon}\Big{(}| \delta_{tt}e^{n+1}+R_{3}^{n+1}|,|\frac{e^{n+1}+e^{n}}{2}|\Big{)}\\
    \leq & \frac{L^{2}}{4 \varepsilon^{2} }\|\delta_{tt}e^{n+1}\|^{2}+\frac{L^{2}}{4 \varepsilon^{2}}\|R_{3}^{n+1}\|^{2} +\frac{1}{2}\|\frac{e^{n+1}+e^{n}}{2}\|^{2}.
  \end{align}

  By Taylor expansion, there exist
  $\vartheta^{n+1} \in \big{(}\frac{\phi^{n+1}+\phi^{n}}{2},
  \phi(t^{n+\frac{1}{2}})\big{)}$ such that
\begin{equation}\label{j3}
  \begin{split}
    J_{8}
    =&-\frac{1}{\varepsilon}\Big{(}f(\frac{\phi^{n+1}+\phi^{n}}{2})-f(\phi(t^{n+\frac{1}{2}})),\frac{e^{n+1}+e^{n}}{2}\Big{)}\\
    =&-\frac{1}{\varepsilon}\Big{(}f'(\phi(t^{n+\frac{1}{2}})) \Big{(}\frac{e^{n+1}+e^{n}}{2}+R_{4}^{n+1}\Big{)},\frac{e^{n+1}+e^{n}}{2}\Big{)}\\
    &-\frac{1}{2\varepsilon}\Big{(}f''(\vartheta^{n+1}) \Big{(}\frac{e^{n+1}+e^{n}}{2}+R_{4}^{n+1}\Big{)}^{2},\frac{e^{n+1}+e^{n}}{2}\Big{)}\\
    \leq
    &-\frac{1}{\varepsilon}\Big{(}f'(\phi(t^{n+\frac{1}{2}}))
    \frac{e^{n+1}+e^{n}}{2},\frac{e^{n+1}+e^{n}}{2}\Big{)}
    +\frac{L_{2}}{\varepsilon}\|\frac{e^{n+1}+e^{n}}{2}\|^{3}_{L^{3}}\\
    &+\frac{1}{
      \varepsilon^{2}}C_{2}\|R_{4}^{n+1}\|^{2}+\frac{1}{2}\|
    \frac{e^{n+1}+e^{n}}{2}\|^{2},
  \end{split}
\end{equation}
where
$ L^{2}+4L^{2}_{2}\| \phi(t)\|_{\infty}^{2} \leq
L^{2}+4L^{2}_{2} C^{2}=:C_{2}$.
For the first term of right hand of (\ref{j3}), we use the spectrum estimate
(\ref{eigen}) to get
\begin{equation}\label{spectrum}
  \varepsilon \|\nabla\frac{e^{n+1}+e^{n}}{2}\|^{2}_{L^{2}}
  +\frac{1}{\varepsilon}\Big{(}f'(\phi(t^{n+\frac{1}{2}}))\frac{e^{n+1}+e^{n}}{2}, \frac{e^{n+1}+e^{n}}{2}\Big{)}
  \geq -C_{0}{\varepsilon}\|\frac{e^{n+1}+e^{n}}{2}\|^{2}.
\end{equation}
Applying (\ref{spectrum}) with a scaling factor
$(1-\varepsilon)$, we get
\begin{equation}\label{spectrum1}
\begin{split}
  &-(1-\varepsilon)\frac{1}{\varepsilon}\Big{(}f'(\phi(t^{n+1}))\frac{e^{n+1}+e^{n}}{2},  \frac{e^{n+1}+e^{n}}{2}\Big{)}\\
  \leq&
  C_{0}{\varepsilon}(1-\varepsilon)\|\frac{e^{n+1}+e^{n}}{2}\|^{2}+(1-\varepsilon)\varepsilon\|\nabla
  \frac{e^{n+1}+e^{n}}{2}\|^{2}.
  \end{split}
\end{equation}
On the other hand,
\begin{equation}\label{tt9}
  -\Big{(}f'(\phi(t^{n+1}))\frac{e^{n+1}+e^{n}}{2},  \frac{e^{n+1}+e^{n}}{2}\Big{)}
  \leq L\|\frac{e^{n+1}+e^{n}}{2}\|^{2}.
\end{equation}
Now, we estimate the $L^{3}$ term. By interpolating $L^{3}$
between $L^{2}$ and $H^{1}$, we get
\[\|\frac{e^{n+1}+e^{n}}{2}\|_{L^{3}}^{3}\leq K \Big{(}
\|\nabla \frac{e^{n+1}+e^{n}}{2}\|^{\frac{d}{2}}
\|\frac{e^{n+1}+e^{n}}{2}\|^{\frac{6-d}{2}}+
\|\frac{e^{n+1}+e^{n}}{2}\|^{3}\Big{)},\]
where K is a constant independ of $\varepsilon$ and $\tau$.
We continue the estimate by using Young's Gronwall
inequality
\begin{equation}\label{tt10}
  \begin{split}
    &\frac{L_{2}}{\varepsilon}K \Big{(} \|\nabla
    \frac{e^{n+1}+e^{n}}{2}\|^{\frac{d}{2}}
    \|\frac{e^{n+1}+e^{n}}{2}\|^{\frac{6-d}{2}}\Big{)}\\
    \leq &\frac{d}{6} \varepsilon^{\frac{3}{d}} \|\nabla
    \frac{e^{n+1}+e^{n}}{2}\|^{3} +\frac{6-d}{6}\Big{(}
    \frac{L_{2}K}{\varepsilon^{\frac{3}{2}}}\Big{)}^{\frac{6}{6-d}}
    \|\frac{e^{n+1}+e^{n}}{2}\|^{3}.
  \end{split}
\end{equation}
Substituting (\ref{spectrum1}) (\ref{tt9}) (\ref{tt10}) into
(\ref{j3}), we get
\begin{equation}\label{j3-1}
\begin{split}
  J_{8}\leq 
  &(C_{0}{\varepsilon}(1-\varepsilon)+L)\|\frac{e^{n+1}+e^{n}}{2}\|^{2}+(1-\varepsilon)\varepsilon\|\nabla
   \frac{e^{n+1}+e^{n}}{2}\|^{2}\\
  &+\frac{d}{6}\varepsilon^{\frac{3}{d}}\|\nabla\frac{e^{n+1}+e^{n}}{2}\|^{3}
  +\Big{(}\frac{6-d}{6}\Big{(}\frac{L_{2}K}{\varepsilon^\frac{3}{2}}\Big{)}^{\frac{6}{6-d}}+
  \frac{L_2}{\varepsilon}K\Big{)}\|\frac{e^{n+1}+e^{n}}{2}\|^{3}\\
  &+\frac{1}{
    \varepsilon^{2}}C_{2}\|R_{4}^{n+1}\|^{2}+\frac{1}{2}\|
  \frac{e^{n+1}+e^{n}}{2}\|^{2}.
\end{split}
\end{equation}

Substituting $J_{1},\cdots, J_{8}$ into (\ref{cet2}), we have
\begin{equation}\label{j4}
  \begin{split}
    &\frac{1}{2\tau }(\|e^{n+1}\|^{2}-\|e^{n}\|^{2})
    +\frac{A\tau}{2}(\| e^{n+1}\|^{2}-\| e^{n}\|^{2})
    +\varepsilon^2\|\nabla \frac{e^{n+1}+e^{n}}{2}\|^{2}\\
    \leq & \|R_{1}^{n+1}\|^{2} + A^{2}\|
      R_{2}^{n+1}\|^{2}
      +B^{2}\|R_{3}^{n+1}\|^{2}+\varepsilon^{2}\|\Delta
      R_{4}^{n+1}\|^{2} +\frac{L^{2}}{4
        \varepsilon^{2}}\|R_{3}^{n+1}\|^{2} +\frac{C_{2}}{
        \varepsilon^{2}}\|R_{4}^{n+1}\|^{2}\\
    &+\frac{9}{4}\| \frac{e^{n+1}+e^{n}}{2}\|^{2}
    +\Big{(}B^{2}+ \frac{L^{2}}{4 \varepsilon^{2}
    }\Big{)}\|\delta_{tt}e^{n+1}\|^{2}
    +(C_{0}{\varepsilon}(1-\varepsilon)+L)\|\frac{e^{n+1}+e^{n}}{2}\|^{2}\\
    & +\frac{d}{6}\varepsilon^{\frac{3}{d}}\|\nabla
    \frac{e^{n+1}+e^{n}}{2}\|^{3}
    +\Big{(}\frac{6-d}{6}\Big{(}
    \frac{L_{2}K}{\varepsilon^{\frac{3}{2}}}\Big{)}^{\frac{6}{6-d}}
    +\frac{L_2K}{\varepsilon}\Big{)}\|\frac{e^{n+1}+e^{n}}{2}\|^{3}.
   \end{split}
 \end{equation}
 To control the 8th term of the right hand side, we pair
 (\ref{cet}) with $\delta_{t}e^{n+1}$ to get
 \begin{equation}\label{t0}
   \begin{split}
     &\frac{1}{\tau}\|\delta_{t}e^{n+1}\|^{2}
     +\frac{\varepsilon}{2}(\|\nabla e^{n+1}\|^{2}-\| \nabla
     e^{n}\|^{2})
     +A\tau\| \delta_{t}e^{n+1}\|^{2}\\
     &+\frac{B}{2}(\|\delta_{t}e^{n+1}\|^{2}-\|\delta_{t}e^{n}\|^{2}+\|\delta_{tt}e^{n+1}\|^{2})\\
     =&(R_{1}^{n+1},\delta_{t}e^{n+1}) -A(
     R_{2}^{n+1},\delta_{t}e^{n+1})
     -B (R_{3}^{n+1},\delta_{t}e^{n+1})\\
     &+\varepsilon ( \Delta R_{4}^{n+1},\delta_{t}e^{n+1})
     -\frac{1}{\varepsilon}\Big{(}f(\frac{3}{2}\phi^{n}-\frac{1}{2}\phi^{n-1})-f(\phi(t^{n+\frac{1}{2}})),\delta_{t}e^{n+1}\Big{)}\\
     =&:\widetilde{J}_{1}+\widetilde{J}_{2}+\widetilde{J}_{3}+\widetilde{J}_{4}+\widetilde{J}_{5}=:\widetilde{J},\ \ \ \rm{n\geq1}.\\
   \end{split}
 \end{equation}
 Analogously, applying the method for $J_{1}, \cdots, J_{4}$
 to $\widetilde{J}_{1}, \cdots, \widetilde{J}_{4}$, yields
 \begin{align}
 \label{t1}
     \widetilde{J}_{1}=&(R_{1}^{n+1},\delta_{t}e^{n+1}) \leq
     \varepsilon\|R_{1}^{n+1}\|^{2}+\frac{1}{4\varepsilon}\|\delta_{t}e^{n+1}\|^{2},\\
\label{t2}
     \widetilde{J}_{2}=&-A( R_{2}^{n+1}, \delta_{t}e^{n+1})
     \leq A^{2}\varepsilon\|
     R_{2}^{n+1}\|^{2}+\frac{1}{4\varepsilon}\|\delta_{t}e^{n+1}\|^{2},\\
 \label{t3}
     \widetilde{J}_{3}=&B (R_{3}^{n+1}, \delta_{t}e^{n+1})
     \leq B^{2}\varepsilon\|
     R_{3}^{n+1}\|^{2}+\frac{1}{4\varepsilon}\|\delta_{t}e^{n+1}\|^{2},\\
 \label{t4}
     \widetilde{J}_{4}=&\varepsilon ( \Delta
     R_{4}^{n+1},\delta_{t}e^{n+1}) \leq
     \varepsilon^{3}\|\Delta
     R_{4}^{n+1}\|^{2}+\frac{1}{4\varepsilon}\|\delta_{t}e^{n+1}\|^{2}.
 \end{align}
 For $\widetilde{J}_{5}$ of (\ref{t0}), we have
  \begin{equation}\label{t5}
    \begin{split}
      \widetilde{J}_{5}=
      &-\frac{1}{\varepsilon}\Big{(}f(\frac{3}{2}\phi^{n}-\frac{1}{2}\phi^{n-1})-f(\phi(t^{n+\frac{1}{2}})),
      \delta_{t}e^{n+1}\Big{)}\\
      \leq &
      -\frac{1}{\varepsilon}\Big{(}f'(\xi^{n+1})\Big{(}
      -\frac{1}{2}\delta_{tt} e^{n+1}-
      \frac{1}{2}R_{3}^{n+1}
      +\frac{e^{n+1}+e^{n}}{2}+R_{4}^{n+1}\Big{)},\delta_{t}e^{n+1} \Big{)}\\
      \leq & \frac{L}{2\varepsilon}\Big{(}
      \frac{3}{2}\|\delta_{t} e^{n+1}\|^{2} +\|\delta_{t}
      e^{n}\|^{2} \Big{)} +
      \frac{L^{2}}{4\varepsilon}\| R_{3}^{n+1}\|^{2}
      +\frac{L^{2}}{ \varepsilon}\| R_{4}^{n+1}\|^{2}
      +\frac{1}{2\varepsilon}\|\delta_{t}e^{n+1}\|^{2}\\
      &+\frac{1}{
        4}\|\frac{e^{n+1}+e^{n}}{2}\|^2+\frac{L^2}{\varepsilon^2}\|\delta_{t}e^{n+1}\|^2.\\
   \end{split}
 \end{equation}

 Substituting $\widetilde{J}_{1},\cdots, \widetilde{J}_{5}$
 into (\ref{t0}), we have
 \begin{equation}\label{t6}
   \begin{split}
     &\frac{1}{\tau}\|\delta_{t}e^{n+1}\|^{2}
     +\frac{\varepsilon}{2}(\|\nabla e^{n+1}\|^{2}-\| \nabla
     e^{n}\|^{2})
     +\frac{L}{2\varepsilon}
     (\|\delta_{t}e^{n+1}\|^{2}-\|\delta_{t}e^{n}\|^{2})\\
     &+A\tau\| \delta_{t}e^{n+1}\|^{2}
     +\frac{B}{2}(\|\delta_{t}e^{n+1}\|^{2}-\|\delta_{t}e^{n}\|^{2}+\|\delta_{tt}e^{n+1}\|^{2})\\
     \leq& \varepsilon\|R_{1}^{n+1}\|^{2}
     +A^{2}\varepsilon\| R_{2}^{n+1}\|^{2}
     +B^{2}\varepsilon\| R_{3}^{n+1}\|^{2}+
     \varepsilon^{3}\|\Delta R_{4}^{n+1}\|^{2}
     +\frac{L^{2}}{4\varepsilon}\| R_{3}^{n+1}\|^{2}\\
     &+\frac{L^{2}}{ \varepsilon}\| R_{4}^{n+1}\|^{2}
     +\Big{(}\frac{5L}{4\varepsilon}+\frac{3}{2\varepsilon}
     + \frac{L^2}{
       \varepsilon^2}\Big{)}\|\delta_{t}e^{n+1}\|^{2}
     +\frac{1}{4}\|\frac{e^{n+1}+e^{n}}{2}\|^2.
   \end{split}
 \end{equation}
 By combining (\ref{j4}) and (\ref{t6}), we get
 \begin{equation}\label{t7-1}
   \begin{split}
     &\frac{1}{2\tau }(\|e^{n+1}\|^{2}-\|e^{n}\|^{2})+
     \frac{\varepsilon}{2}(\|\nabla e^{n+1}\|^{2}-\| \nabla
     e^{n}\|^{2})
     +\frac{A\tau}{2}(\| e^{n+1}\|^{2}-\| e^{n}\|^{2})\\
     & +\Big{(}\frac{ B}{2}+2B^{2}+\frac{L}{2\varepsilon}+
     \frac{L^{2}}{ 2\varepsilon^{2}}\Big{)}(\|\delta_{t}e^{n+1}\|^{2} - \|\delta_{t}e^{n}\|^{2})\\
     &+A\tau\|
     \delta_{t}e^{n+1}\|^{2}+\frac{B}{2}\|\delta_{tt}e^{n+1}\|^{2}
     +\frac{1}{\tau}\|\delta_{t}e^{n+1}\|^{2}
     +\varepsilon^{2}\|\nabla \frac{e^{n+1}+e^{n}}{2}\|^{2}\\
     \leq & (1+\varepsilon)\left(\|R_{1}^{n+1}\|^{2} +
       A^{2}\| R_{2}^{n+1}\|^{2}
       +B^{2}\|R_{3}^{n+1}\|^{2}+\varepsilon^{2}\|\Delta
       R_{4}^{n+1}\|^{2}
       +\frac{L^{2}}{4 \varepsilon^{2}}\|R_{3}^{n+1}\|^{2} \right)\\
     &+\left( \frac{C_{2}}{ \varepsilon^{2}}+\frac{L^{2}}{
         \varepsilon} \right) \| R_{4}^{n+1}\|^{2}
     +\Big{(}4B^{2}+2 \frac{L^{2}}{ \varepsilon^{2} }
     +\frac{5L}{4\varepsilon}
     +\frac{3}{2\varepsilon} \Big{)}\|\delta_{t}e^{n+1}\|^{2}\\
     &+\Big{(}\frac{5}{2}+
     C_{0}{\varepsilon}(1-\varepsilon)+L\Big{)}\|\frac{e^{n+1}+e^{n}}{2}\|^{2}
     +G^{n+1}\|\frac{e^{n+1}+e^{n}}{2}\|^{2}\\
     &+Q^{n+1}\|\nabla \frac{e^{n+1}+e^{n}}{2}\|^{2},
   \end{split}
 \end{equation}
 where
 $Q^{n+1}=\frac{d}{6} \varepsilon^{\frac{3}{d}}\|\nabla
 \frac{e^{n+1}+e^{n}}{2}\|$,
 $G^{n+1}=\Big{(}\frac{6-d}{6}\Big{(}
 \frac{L_{2}K}{\varepsilon^{\frac{3}{2}}}
 \Big{)}^{\frac{6}{6-d}}+\frac{L_2K}{\varepsilon}\Big{)}
 \|\frac{e^{n+1}+e^{n}}{2}\|$.
 Taking
 $\tau \leq 1/\Big{(}4B^{2}+ 2\frac{L^{2}}{ \varepsilon^{2}
 } +\frac{5L}{4\varepsilon}+\frac{3}{2\varepsilon} \Big{)}
 $,
 if $Q^{n+1}$ is uniformly bounded by constant
 $\frac{\varepsilon^{2}}{2}$, $G^{n+1}$ is uniformly bounded
 by constant $\frac{1}{2}$,
 then by Gronwall inequality, we get the finer error estimate (\ref{error-2}).

 We prove this by induction. Assuming that the finer
 estimate (\ref{error-2}) holds for all first $N$ time
 steps, the coarse estimate \eqref{cet2-10} leads to
 \begin{equation}\label{t8}
   \begin{split}
     &\|e^{N+1}\|^{2} +2\tau \varepsilon\|\nabla
     \frac{e^{N+1}+e^{N}}{2}\|^{2} +A\tau^{2}\|e^{N+1}\|^{2}
     +B\tau \|e^{N+1}\|^{2}\\
     \lesssim& \varepsilon^{-\max\{4\sigma_1+11,
       \sigma_{0}\}}\tau^{4}.
   \end{split}
 \end{equation}
 Then, if 
 $\tau \lesssim \varepsilon^{\frac{1}{3}\max\{4\sigma_1+11,
 	\sigma_{0}\}+\frac{5}{3} -\frac{2}{d}}$,
 we have
  \begin{align}\label{t9-1}
   Q^{N+1} \lesssim
   \varepsilon^{\frac{3}{d}}\varepsilon^{-\frac{1}{2}\max\{4\sigma_1+11, \sigma_{0}\}-\frac{1}{2}}\tau^{\frac{3}{2}}
   \lesssim \frac{\varepsilon^{2}}{2}.
 \end{align}
 If $\tau \lesssim\varepsilon^{\frac{1}{4}\max\{4\sigma_1+11,
 	\sigma_{0}\}+ \frac{9}{2(6-d)}}$,
 we have
 \begin{align}\label{t10-1}
  G^{N+1}
  \lesssim  \varepsilon^{-\frac{9}{6-d}}\varepsilon^{-\frac{1}{2}\max\{4\sigma_1+11, \sigma_{0}\}}\tau^{2}
  \lesssim \frac{1}{2}.
\end{align}
By taking $\tau$ satisfies inequality \eqref{error-1}, we get the finer error estimate for $N+1$ step, and
the proof is completed by mathematical induction.  \hfill
\end{proof}


\begin{remark}
  Theorem \ref{CNerror} and \ref{BDFerror} are valid for the
  special cases i) $A=0$, ii) $B=0$, iii) both $A=0$ and
  $B=0$, since the condition \eqref{eq:BDF:ABcond} and
  \eqref{accn2} are not used in the proof. On the other
  hand, in Theorem \ref{CNerror} and \ref{BDFerror}, the
  step size need be smaller than $\varepsilon^2$ to
  guarantee the convergence, which is much stronger than the
  requirement for the {unstabilized} schemes {(i.e. the case $A=B=0$)} to be 
  stable.
\end{remark}
 
 \yu{
 \begin{remark}\label{rmk:errest}
 	The proofs of \ref{BDFerror} and Theorem \ref{CNerror} 
 	are inspired by the works \cite{feng_error_2004}, \cite{kessler_posteriori_2004}, \cite{feng_analysis_2015} and \cite{feng_analysis_2016} for first order convex splitting schemes. 
 	The main difference is that we use a mathematical induction to
 	handle high order terms come from the $L^3$ term, while  
 	a generalized Gronwall lemma is used in \cite{feng_analysis_2015}, \cite{feng_analysis_2016}, and a continuation argument is used in 
 	\cite{kessler_posteriori_2004}.
 \end{remark}
}

 \new{
	\begin{remark}\label{rmk:gamma}
		For the case that $\gamma=O(1/\varepsilon)$, 
		we can get similar second order convergence results with the constant 
		does not depend on $1/\varepsilon$ exponentially for both SL-BDF2 and 
		SL-CN schemes. Take the SL-BDF2 scheme 
		as an example. By using a Cauchy inequality with $\varepsilon$, one can 
		put an $\varepsilon$ in front of the $\|e^{n+1}\|^2$ terms in  
		\eqref{BDF6} \eqref{BDF7}, \eqref{BDF8}, \eqref{J4} and \eqref{J6}. 
		Then, 
		by replacing the factor 
		$(1-\varepsilon)$ in \eqref{spectrum1BDF} with $1-\varepsilon^2$, and 
		multiplying \eqref{BDF12} by $\varepsilon$, we can get an estimate
		similar to \eqref{finer-error} for time steps small enough, but 
		the exponential factor now scales like 
		$\exp(O(\varepsilon)T)$.
	\end{remark}
}

\section{Implementation and numerical results}

In this section, we numerically verify our schemes are
second order accurate in time and energy stable.

We use the commonly used double-well potential
$F(\phi)=\frac{1}{4}(\phi^{2}-1)^{2}$. Since the exact
solution satisfies the maximum principle $|\phi|\leq 1$, it
is a common practice to modify $F(\phi)$ to have a quadratic
growth for $|\phi|>1$, such that a global Lipschitz
condition is satisfied
(cf. e.g. \cite{shen_numerical_2010},\cite{condette_spectral_2011}). To
get a $C^4$ smooth double-well potential with quadratic
growth, we introduce
$\tilde{F}(\phi) \in C^{\infty}(\mathbf{R})$ as a smooth
mollification of
\begin{equation}\label{efnew}
\hat{F}(\phi)=
\begin{cases}
  \frac{11}{2}(\phi-2)^{2}+6(\phi-2)+\frac94, &\phi>2, \\
  \frac{1}{4 }(\phi^{2}-1)^{2}, &\phi\in [-2,2], \\
  \frac{11}{2}(\phi+2)^{2}+6(\phi+2)+\frac94, &\phi<-2.
\end{cases}
\end{equation}
with a mollification parameter much smaller than 1, to
replace $F(\phi)$. Note that the truncation points $-2$ and
$2$ used here are for convenience only. Other values outside
of region $[-1,1]$ can be used as well.  For simplicity, we
still denote the modified function $\tilde{F}$ by $F$.

\subsection{Space discrete and implementation}
To test the numerical scheme, we solve \eqref{eq:AC} in a
2-dimensional domain $\Omega=[-1,1]^2$ and a 3-dimensional domain
$\Omega=[-1,1]^3$. We use a Legendre
Galerkin method similar as in \cite{shen_efficient_2015,
  yu_numerical_2017} for spatial discretization. For example, we
define
\[
V_M = \mbox{span}\{\,\varphi_k(x)\varphi_j(y) \varphi_i(z),\
k,j,i=0,\ldots, M-1\,\}\in H^1(\Omega),
\]
as Galerkin approximation space for $\phi^{n+1}$ in 3-dimensional case.
Here
$\varphi_0(x) = L_0(x); \varphi_1(x)=L_1(x);
\varphi_k(x)=L_k(x)-L_{k+2}(x), k=2,\ldots, M\!-\!1 $.
$L_k(x)$ denotes the Legendre polynomial of degree $k$.
Then the full discretized form for the SL-BDF2 scheme reads:

Find $(\phi^{n+1}, \mu^{n+1}) \in (V_M)^{2}$ such that
\begin{equation}\label{eq:bdf2:fulldis1}
\begin{split}
  \frac{1}{2\tau \gamma}(3\phi^{n+1}-4\phi^{n}+\phi^{n-1},
  \omega) ={} &-\varepsilon (\nabla \phi^{n+1},\nabla
  \varphi)
  -\frac{1}{\varepsilon}(f(2\phi^{n}-\phi^{n-1}),\varphi)\\
  &-A\tau ( \delta_{t}\phi^{n+1}, \varphi)
  -B(\delta_{tt}\phi^{n+1},\varphi), \quad \forall\, \varphi
  \in V_M.
\end{split}
\end{equation}
This is a linear system with constant coefficients for
$\phi^{n+1}$, which can be efficiently solved.  We use a
spectral transform with double quadrature points to
eliminate the aliasing error and efficiently evaluate the
integration $(f(2\phi^n-\phi^{n-1}),\varphi)$ in equation
\eqref{eq:bdf2:fulldis1}.

Given $\phi^0$, to start the second order schemes, we use
following first order stabilized scheme with smaller time steps to generate
$\phi^1$, 
 \begin{equation}\label{phi1}
   \frac{\varphi^{n+1}-\varphi^{n}}{\tau \gamma} =\varepsilon \Delta \varphi^{n+1}
   -\frac{1}{\varepsilon}f(\varphi^{n})-A \delta_{t}\varphi^{n+1},\
   n=0,\ldots,m-1,
 \end{equation}
where $\varphi^0=\phi^0$, $\phi^1=\varphi^m$.

 We take $\varepsilon=0.075$ and $M=63$ and use 
random initial values $\phi_0$ to test the
 stability and accuracy of the proposed schemes. For the
 3-dimensional case, the initial value is given as
  $\{\phi_0(x_i,y_j,z_k) \}\in\,
  {\bf{R}}^{2M\times2M\times2M}$
  with $x_i, y_j, z_k$ are tensor product Legendre-Gauss
  quadrature points and $\phi_0(x_i,y_j,z_k)$ is a uniformly
  distributed random number between $-1$ and $1$ (shown in
  the first picture of Fig. \ref{fig1});

\begin{figure}[H]
	\centering
	\includegraphics[width=0.32\textwidth]{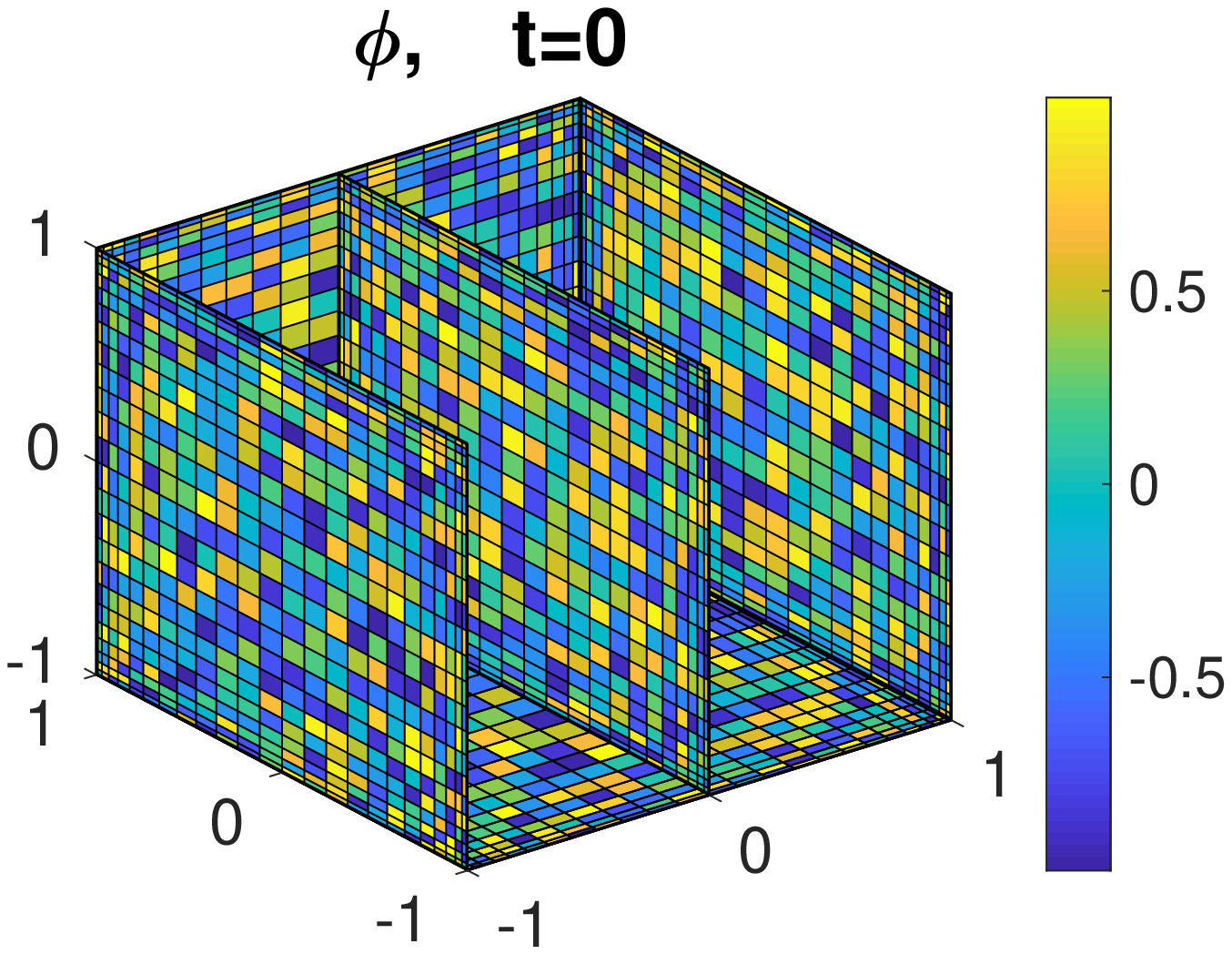}
	\includegraphics[width=0.32\textwidth]{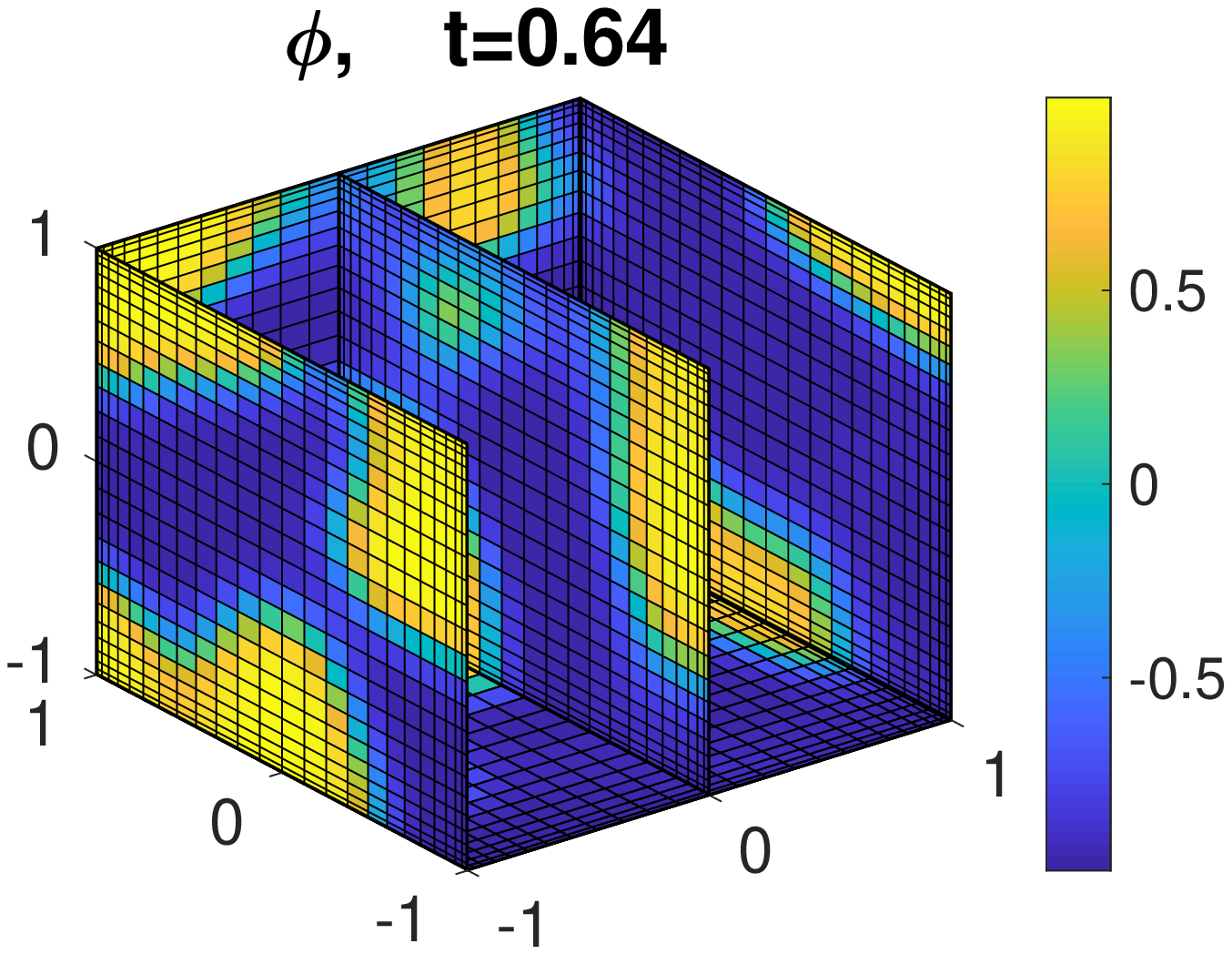}
	\includegraphics[width=0.32\textwidth]{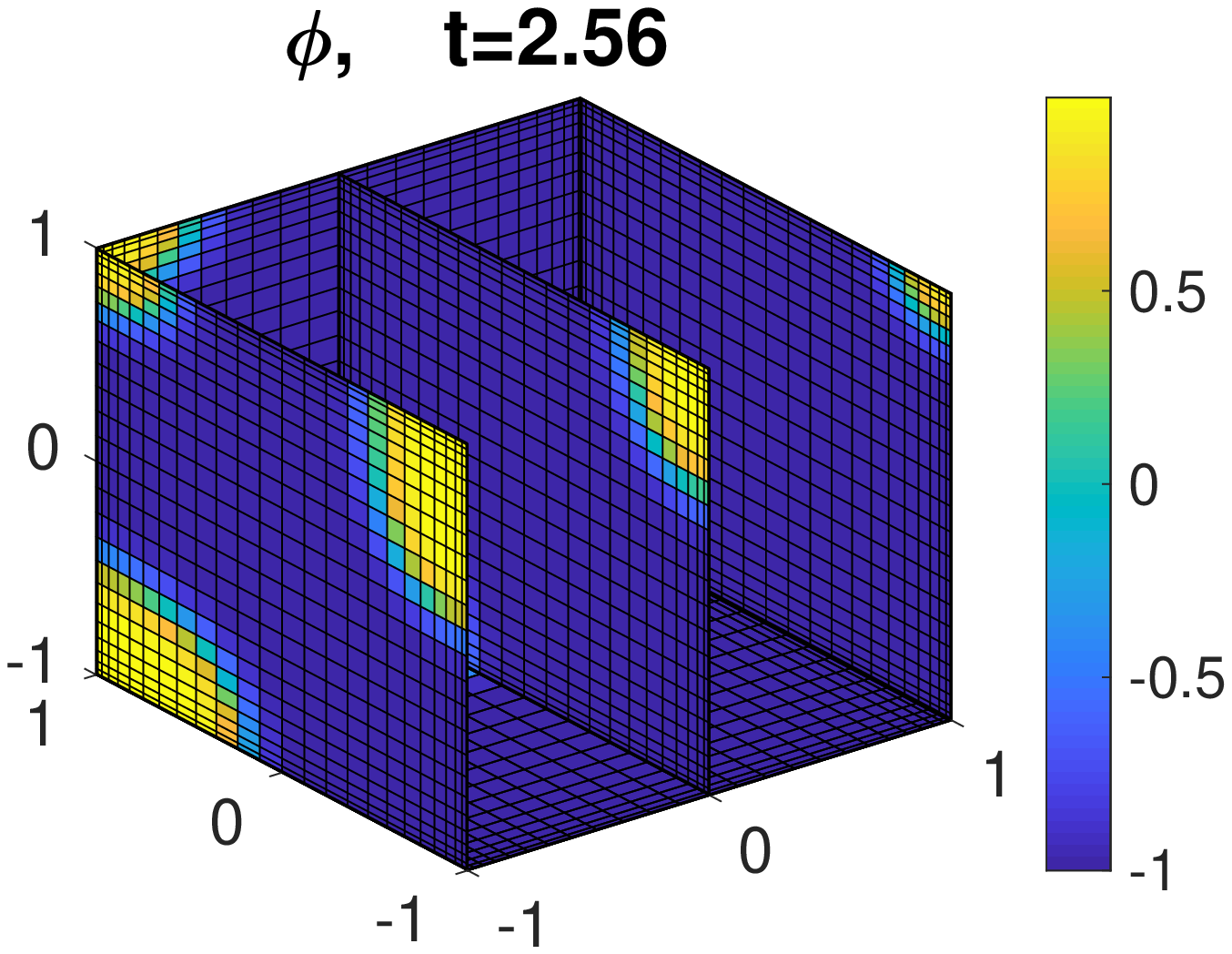}
  	\caption{The view of five different slices of the initial value $\phi_0$ and the corresponding solution. 
  		First) the random initial values $\phi_0$; 
  		Second) $\phi_1$, the solution at $t=0.64$ of the Allen-Cahn equation with initial value $\phi_0$;  
  		Third) the solution at $t=2.56$ of the
  		the Allen-Cahn equation with initial value $\phi_0$.  Equation parameter $\gamma=1$, $\varepsilon=0.075$. }
  	\label{fig1}
\end{figure}

\subsection{Stability results}

We present the required minimum values of $A$ (resp. $B$)
with different $B$ (resp. $A$) and $\tau$ values for stably
solving the Allen-Cahn equation \eqref{eq:AC} in
3-dimensional case in Table \ref{tstab3d:B}
(resp. \ref{tstab3d:A}). Here by ``stably solving", we mean
the energy keep dissipating in first 1024 time steps. The
corresponding 2-dimensional results are given in Table
\ref{tstab2d:B} and \ref{tstab2d:A}.  Those results are
obtained by using initial value $\phi_0$, the results for
the cases taking initial value $\phi_1$ are similar.  From
those tables, we see that the maximum required A values are
of order $O(\varepsilon^2)$ and the maximum required B
values are of order $O(\varepsilon)$. For $\tau$ small
enough, both schemes are stable with $A=0$ and $B=0$. This
is consistent to our analysis result.  On the other hand
side, by using a nonzero $B$, e.g. $B=10$, the requirement
for a large $A$ will be dramatically reduced.

To check the energy dissipation property, we present in
Figure \ref{fig:2Energy} the log-log plot of the energy
versus time for two schemes using different time step-sizes.
We see the energy decaying property is maintained.

\begin{table}[htpb]
 \begin{tabular}{|c|c|c|c|c|c|c|c|c|}
   \hline
   \multirow{3}{*}{$\tau$ }
   &  \multicolumn{3}{|c|}{SL-BDF2}
   &  \multicolumn{3}{|c|}{SL-CN}
   \\ \cline{2-7}  & $B=0$ & $B=5$ & $B=10$  & $B=0$ & $B=5$ & $B=10$ \\ \hline
   10   & 3   & 2   & 1  & 3  &  2  & 1  \\ \hline
   1    & 30  & 20  & 10 & 30 &  20 & 10  \\ \hline
   0.1  & 100 &  0  &  0 & 200& 100  & 0  \\ \hline
   0.01 & 0   & 0   & 0  & 0  & 0  & 0  \\ \hline
\end{tabular}
 \centering
 \caption{The minimum values of $A$ (only values  
 	$\{ 0,1, \ldots,5\}\cup \{10, 15, \ldots,50\}\cup\{100, 150, \ldots,500 \}$ 
 	are tested for A)
   to make SL-BDF2 and SL-CN scheme stable when $B$ and $\tau$ taking different
   values. The results are from 3-dimensional simulations with $\gamma=1$, $\varepsilon=0.075$.}
\label{tstab3d:B}
\end{table}

\begin{table}[htpb]
 \begin{tabular}{|c|c|c|c|c|c|c|c|c|}
   \hline
   \multirow{3}{*}{$\tau$ }
   &  \multicolumn{3}{|c|}{SL-BDF2}
   &  \multicolumn{3}{|c|}{SL-CN}
   \\ \cline{2-7}  & $A=0$ & $A=10$ & $A=20$ & $A=0$ & $A=10$ & $A=20$\\ \hline
   10   & 20 & 0   &  0  &   30  &  0  & 0  \\ \hline
   1    & 20 & 10  &  5  &   30  &  10  & 3 \\ \hline
   0.1  & 10 & 10  &  10  &   10  &  10  & 10  \\ \hline
   0.01 & 0  & 0   &  0  &   0  &  0  &  0 \\ \hline
 \end{tabular}
 \centering
 \caption{The minimum values of $B$ (only values 
   $\{ 0,1, \ldots, 5\}\cup \{10,15, \ldots, 50 \}$ 
   are tested for B)
   to make scheme SL-BDF2 and SL-CN stable when  $A$ and $\tau$ taking different
   values. The results are from 3-dimensional simulations with $\gamma=1$, $\varepsilon=0.075$.}
\label{tstab3d:A}
\end{table}

\begin{table}[htpb]
 \begin{tabular}{|c|c|c|c|c|c|c|c|c|}
   \hline
   \multirow{3}{*}{$\tau$ }
   &  \multicolumn{3}{|c|}{SL-BDF2}
   &  \multicolumn{3}{|c|}{SL-CN}
   \\ \cline{2-7}  & $B=0$ & $B=5$ & $B=10$  & $B=0$ & $B=5$ & $B=10$ \\ \hline
   10   & 3   & 2   & 1  & 4  &  2  & 1  \\ \hline
   1    & 35  & 20  & 20 & 30 &  20 & 10  \\ \hline
   0.1  & 100 &  10 &  0 & 200& 100  & 0  \\ \hline
   0.01 & 0   & 0   & 0  & 0  & 0  & 0  \\ \hline
 \end{tabular}
 \centering
 \caption{The minimum values of $A$ (only values  
 	$\{ 0,1, \ldots,5\}\cup \{10, 15, \ldots,50\}\cup\{100, 150, \ldots,500 \}$ 
 	are tested for A)
   to make SL-BDF2 and SL-CN scheme stable when $B$ and $\tau$ taking different
   values. The results are from 2-dimensional simulations with $\gamma=1$, $\varepsilon=0.075$.}
 \label{tstab2d:B}
\end{table}

\begin{table}[htpb]
  \begin{tabular}{|c|c|c|c|c|c|c|c|c|}
    \hline
    \multirow{3}{*}{$\tau$ }
    &  \multicolumn{3}{|c|}{SL-BDF2}
    &  \multicolumn{3}{|c|}{SL-CN}
    \\ \cline{2-7}  & $A=0$ & $A=10$ & $A=20$ & $A=0$ & $A=10$ & $A=20$\\ \hline
    10   & 20 & 0   &  0  &   20  &  0  & 0  \\ \hline
    1    & 20 & 15  &  0  &   20  &  10  & 3 \\ \hline
    0.1  & 10 & 5  &  2  &   10  &  10  & 10  \\ \hline
    0.01 & 0  & 0   &  0  &   0  &  0  &  0 \\ \hline
  \end{tabular}
  \centering
  \caption{The minimum values of $B$ (only values 
    $\{ 0,1, \ldots, 5\}\cup \{10,15, \ldots, 50 \}$ 
    are tested for B)
    to make scheme SL-BDF2 and SL-CN stable when  $A$ and $\tau$ taking
    different values. The results are from 2-dimensional simulations with $\gamma=1$, $\varepsilon=0.075$.}
  \label{tstab2d:A}
\end{table}

\begin{figure}[htpb]
	\centering
	\includegraphics[width=0.48\textwidth]{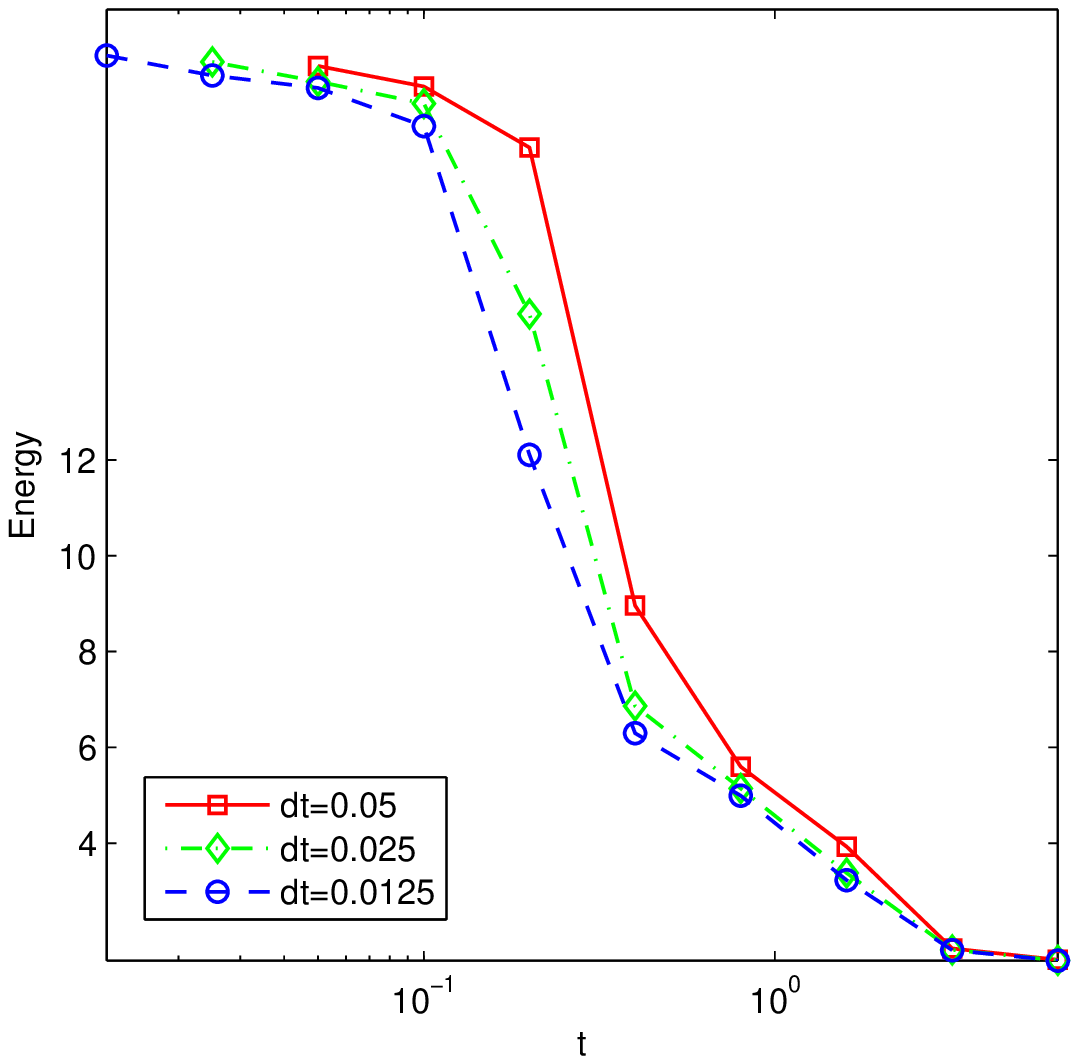}
	\includegraphics[width=0.48\textwidth]{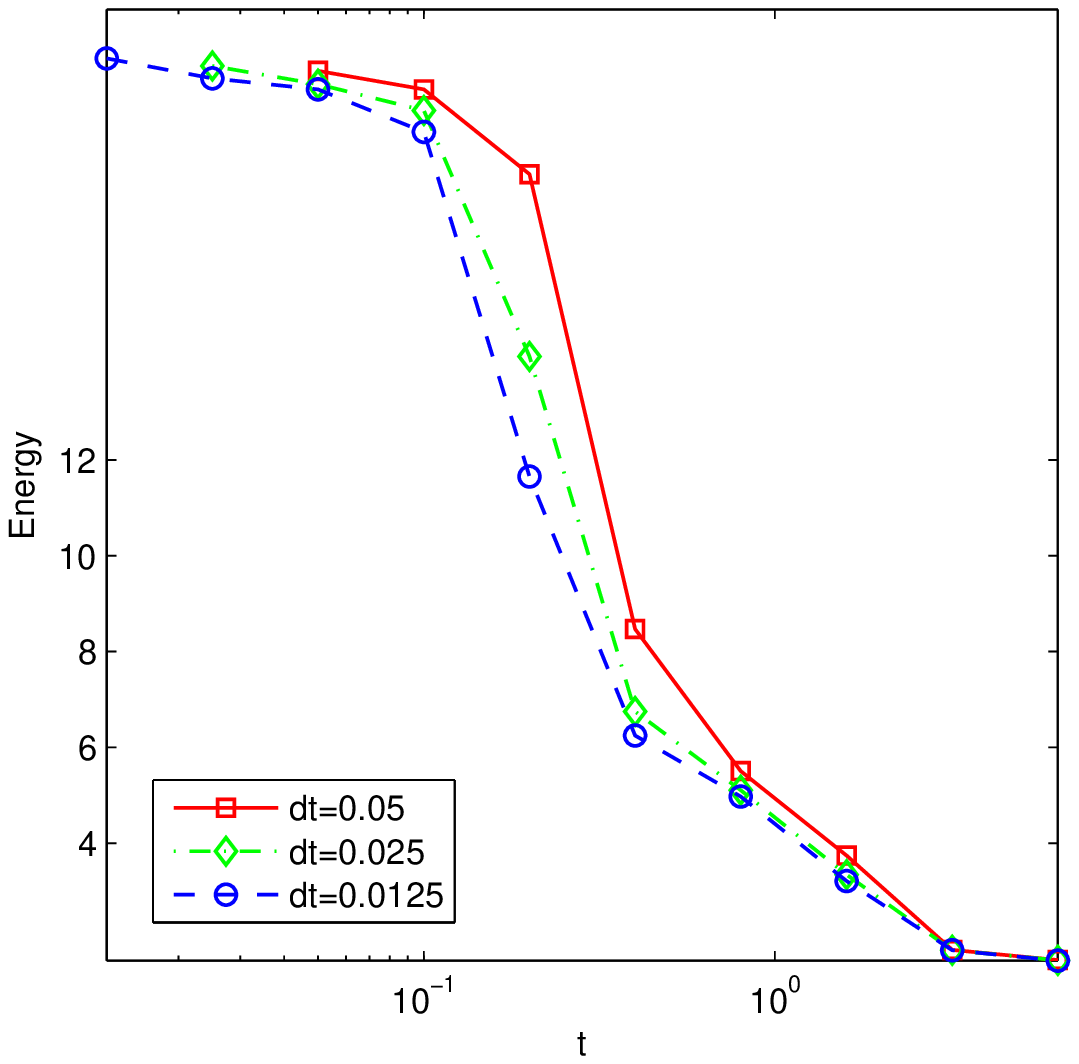}
	\caption{The discrete energy dissipation of the two
      schemes solving the Allen-Cahn equation with initial
      value $\phi_1$, and parameter $\gamma=1$,
      $\varepsilon=0.075$.  Left) result of SL-BDF2 scheme;
      Right) result of SL-CN scheme.}
	\label{fig:2Energy}
\end{figure}

\subsection{Accuracy results}
We take initial value $\phi_1$ (see the second plot in Fig.~\ref{fig1}) for \eqref{eq:AC} to test the
accuracy of the two schemes in a 2-dimensional domain
$\Omega=[-1,1]\times[-1,1]$.  The Allen-Cahn equation with
the time relaxation parameter $\gamma=0.5$ are solved from
$t=0$ to $T=1.28$. To calculate the numerical error, we use
the numerical result generated using $\tau=10^{-4}$ as a
reference of exact solution.  The results are given in Table
\ref{tbl:conv:SL-BDF2} and Table \ref{tbl:conv:SL-CN}. We
see that the schemes are second order accuracy in both $L^2$
and $H^1$ norm.

\begin{table}[htpb]
  \begin{tabular}{|c|c|c|c|c|c|c|}
    \hline
    $\tau$ &  $L^2$ Error & Order & $H^1$ Error & Order\\ \hline
    0.032 & 3.21 E-01 &       & 3.41 & \\ \hline
    0.016 &  1.19E-01 & 1.425 & 1.49 &  1.189\\ \hline
    8E-3 &  2.90E-02 & 2.043 & 3.68E-01 & 2.021 \\ \hline
    4E-3 &  7.15E-03 & 2.019 & 8.91E-02 & 2.047 \\ \hline
    2E-3 &  1.82E-03 & 1.976 & 2.26E-02 & 1.980 \\ \hline
    1E-3 &  4.50E-04 & 2.016 & 5.58E-03 & 2.016 \\ \hline
  \end{tabular}
  \centering
  \caption{The convergence of the SL-BDF2 scheme
    with $B=5$, $A=10$
    for the Allen-Cahn equation with initial value $\phi_1$, parameter $\gamma=0.5, \varepsilon=0.05$.
    The errors are calculated at $T=1.28$.}
  \label{tbl:conv:SL-BDF2}
\end{table}

\begin{table}[htpb]
  \begin{tabular}{|c|c|c|c|c|c|c|}
    \hline
    $\tau$ &  $L^2$ Error & Order & $H^1$ Error & Order\\ \hline
    0.032 &  2.84E-01 &      & 3.11 &  \\ \hline
    0.016 &  9.71E-02 &1.548 & 1.25 &     1.320    \\ \hline
    8E-3 &  2.19E-02 & 2.152 & 2.76E-01 & 2.178 \\ \hline
    4E-3 &  5.33E-03 & 2.035 & 6.63E-02 & 2.055 \\ \hline
    2E-3 &  1.34E-03 & 1.997 & 1.66E-02 & 2.000 \\ \hline
    1E-3 &  3.27E-04 & 2.031 & 4.06E-03 & 2.032 \\ \hline
\end{tabular}
 \centering
 \caption{The convergence of the SL-CN scheme
   with 
   $B=5$, $A=10$
   for the Allen-Cahn equation with initial value $\phi_1$, parameter $\gamma=0.5, \varepsilon=0.05$.
   The errors are calculated at $T=1.28$.}
\label{tbl:conv:SL-CN}
\end{table}

\section{Conclusions}

We proposed two second order stabilized linear schemes,
namely the SL-BDF2 and the SL-CN scheme, for the phase-field
Allen-Cahn equation. In both schemes, the nonlinear bulk
forces are treated explicitly with two additional linear
stabilization terms to guarantee unconditionally energy
stable.  The schemes lead to linear systems with constant
coefficients thus can be efficiently solved. An optimal
error estimate is given by using a spectrum argument to
remove the exponential dependence on $1/\varepsilon$.  The
error analysis also holds for the special cases when one of
the stabilization constants or both of them take zero
values.  Numerical results verified the stability and
accuracy of the proposed schemes.

\section*{Acknowledgment}
This work is partially supported by NNSFC Grant 11771439, 11371358 and Major Program of NNSFC under Grant 91530322.
The authors would like to thank Prof. Jie Shen for helpful discussions.

\bibliographystyle{siam}  
\bibliography{LSS2}

\begin{thebibliography}{10}

\bibitem{allen_microscopic_1979}
{\sc S.~M. Allen and J.~W. Cahn}, {\em A microscopic theory for antiphase
  boundary motion and its application to antiphase domain coarsening}, Acta
  Metall. Mater., 27 (1979), pp.~1085--1095.

\bibitem{baskaran_energy_2013}
{\sc A.~Baskaran, P.~Zhou, Z.~Hu, C.~Wang, S.~M. Wise, and J.~S. Lowengrub},
  {\em Energy stable and efficient finite-difference nonlinear multigrid
  schemes for the modified phase field crystal equation}, J. Comput. Phys., 250
  (2013), pp.~270--292.

\bibitem{benesova_implicit_2014}
{\sc B.~Benesov\'a, C.~Melcher, and E.~S\"uli}, {\em An implicit midpoint
  spectral approximation of nonlocal {Cahn}--{Hilliard} equations}, SIAM J.
  Numer. Anal., 52 (2014), pp.~1466--1496.

\bibitem{cahn_free_1958}
{\sc J.~W. Cahn and J.~E. Hilliard}, {\em Free energy of a nonuniform system.
  {I}. interfacial free energy}, J. Chem. Phys., 28 (1958), pp.~258--267.

\bibitem{chen_applications_1998}
{\sc L.~Chen and J.~Shen}, {\em Applications of semi-implicit
  {Fourier}-spectral method to phase field equations}, Comput. Phys. Commun.,
  108 (1998), pp.~147--158.

\bibitem{chen_linear_2014}
{\sc W.~Chen, C.~Wang, X.~Wang, and S.~M. Wise}, {\em A linear iteration
  algorithm for a second-order energy stable scheme for a thin film model
  without slope selection}, J. Sci. Comput., 59 (2014), pp.~574--601.

\bibitem{chen_spectrum_1994}
{\sc X.~Chen}, {\em Spectrum for the {Allen-Cahn}, {Cahn-Hillard}, and
  phase-field equations for generic interfaces}, Commun. Part. Diff. Eq., 19
  (1994), pp.~1371--1395.

\bibitem{cheng_second-order_2016}
{\sc K.~Cheng, C.~Wang, S.~M. Wise, and X.~Yue}, {\em A second-order, weakly
  energy-stable pseudo-spectral scheme for the {Cahn}-{Hilliard} equation and
  its solution by the homogeneous linear iteration method}, J. Sci. Comput., 69
  (2016), pp.~1083--1114.

\bibitem{condette_spectral_2011}
{\sc N.~Condette, C.~Melcher, and E.~S\"uli}, {\em Spectral approximation of
  pattern-forming nonlinear evolution equations with double-well potentials of
  quadratic growth}, Math. Comp., 80 (2011), pp.~205--223.

\bibitem{de_mottoni_evolution_1989}
{\sc P.~de~Mottoni and M.~Schatzman}, {\em \'evolution g\'eom\'etrique
  d'interfaces. ({Geometrical} evolution of interfaces)}, C. R. Acad. Sci. Sbr.
  I,  (1989).

\bibitem{de_mottoni_geometrical_1995}
\leavevmode\vrule height 2pt depth -1.6pt width 23pt, {\em Geometrical
  evolution of developed interfaces}, T. Am. Math. Soc., 347 (1995),
  pp.~1533--1589.

\bibitem{diegel_stability_2016}
{\sc A.~E. Diegel, C.~Wang, and S.~M. Wise}, {\em Stability and convergence of
  a second order mixed finite element method for the {Cahn}-{Hilliard}
  equation}, IMA J. Numer. Anal., 36 (2016), pp.~1867--1897.

\bibitem{du_numerical_1991}
{\sc Q.~Du and R.~A. Nicolaides}, {\em Numerical analysis of a continuum model
  of phase transition}, SIAM J. Numer. Anal., 28 (1991), pp.~1310--1322.

\bibitem{elder_modeling_2004}
{\sc K.~R. Elder and M.~Grant}, {\em Modeling elastic and plastic deformations
  in nonequilibrium processing using phase field crystals}, Phys. Rev. E, 70
  (2004), p.~051605.

\bibitem{elliott_cahnhilliard_1996}
{\sc C.~Elliott and H.~Garcke}, {\em On the {Cahn}-{Hilliard} equation with
  degenerate mobility}, SIAM J. Math. Anal., 27 (1996), pp.~404--423.

\bibitem{elliott_cahn-hilliard_1989}
{\sc C.~M. Elliott}, {\em The {Cahn}-{Hilliard} model for the kinetics of phase
  separation}, in Mathematical {Models} for {Phase} {Change} {Problems}, J.~F.
  Rodrigues, ed., no.~88 in International {Series} of {Numerical}
  {Mathematics}, Birkhäuser Basel, 1989, pp.~35--73.
\newblock DOI: 10.1007/978-3-0348-9148-6\_3.

\bibitem{elliott_global_1993}
{\sc C.~M. Elliott and A.~M. Stuart}, {\em The global dynamics of discrete
  semilinear parabolic equations}, SIAM J. Numer. Anal., 30 (1993),
  pp.~1622--1663.

\bibitem{eyre_unconditionally_1998}
{\sc D.~J. Eyre}, {\em Unconditionally gradient stable time marching the
  {Cahn}-{Hilliard} equation}, in Computational and mathematical models of
  microstructural evolution ({San} {Francisco}, {CA}, 1998), vol.~529 of Mater.
  {Res}. {Soc}. {Sympos}. {Proc}., MRS, 1998, pp.~39--46.

\bibitem{feng_uniquely_2018}
{\sc W.~Feng, Z.~Guan, J.~Lowengrub, C.~Wang, S.~M. Wise, and Y.~Chen}, {\em A
  uniquely solvable, energy stable numerical scheme for the functionalized
  {Cahn-Hilliard} equation and its convergence analysis}, J. Sci. Comput.,
  (2018), pp.~1--30.

\bibitem{feng_fully_2006}
{\sc X.~Feng}, {\em Fully discrete finite element approximations of the
  {Navier}--{Stokes}--{Cahn}-{Hilliard} diffuse interface model for two-phase
  fluid flows}, SIAM J. Numer. Anal., 44 (2006), pp.~1049--1072.

\bibitem{feng_analysis_2015}
{\sc X.~Feng and Y.~Li}, {\em Analysis of symmetric interior penalty
  discontinuous {Galerkin} methods for the {Allen-Cahn} equation and the mean
  curvature flow}, {IMA} J. Numer. Anal., 35 (2015), pp.~1622--1651.

\bibitem{feng_analysis_2016}
{\sc X.~Feng, Y.~Li, and Y.~Xing}, {\em Analysis of mixed interior penalty
  discontinuous {G}alerkin methods for the {Cahn-Hilliard} equation and the
  {Hele-Shaw} flow}, {SIAM} J. Numer. Anal., 54 (2016), pp.~825--847.

\bibitem{feng_numerical_2003}
{\sc X.~Feng and A.~Prohl}, {\em Numerical analysis of the {Allen}-{Cahn}
  equation and approximation for mean curvature flows}, Numer. Math., 94
  (2003), pp.~33--65.

\bibitem{feng_error_2004}
\leavevmode\vrule height 2pt depth -1.6pt width 23pt, {\em Error analysis of a
  mixed finite element method for the {Cahn}-{Hilliard} equation}, Numer.
  Math., 99 (2004), pp.~47--84.

\bibitem{furihata_stable_2001}
{\sc D.~Furihata}, {\em A stable and conservative finite difference scheme for
  the {C}ahn-{H}illiard equation}, Numer. Math., 87 (2001), pp.~675--699.

\bibitem{gomez_provably_2011}
{\sc H.~Gomez and T.~J.~R. Hughes}, {\em Provably unconditionally stable,
  second-order time-accurate, mixed variational methods for phase-field
  models}, J. Comput. Phys., 230 (2011), pp.~5310--5327.

\bibitem{guillen-gonzalez_linear_2013}
{\sc F.~Guillén-González and G.~Tierra}, {\em On linear schemes for a
  {Cahn-Hilliard} diffuse interface model}, J. Comput. Phys., 234 (2013),
  pp.~140--171.

\bibitem{guillen-gonzalez_second_2014}
{\sc F.~Guillén-González and G.~Tierra}, {\em Second order schemes and
  time-step adaptivity for {Allen}-{Cahn} and {Cahn}-{Hilliard} models},
  Comput. Math. Appl., 68 (2014), pp.~821--846.

\bibitem{guo_h2_2016}
{\sc J.~Guo, C.~Wang, S.~M. Wise, and X.~Yue}, {\em An {$H^2$} convergence of a
  second-order convex-splitting, finite difference scheme for the
  three-dimensional {Cahn}-{Hilliard} equation}, Commun. Math. Sci, 14 (2016),
  pp.~489--515.

\bibitem{han_numerical_2017}
{\sc D.~Han, A.~Brylev, X.~Yang, and Z.~Tan}, {\em Numerical analysis of second
  order, fully discrete energy stable schemes for phase field models of two
  phase incompressible flows}, J. Sci. Comput., 70 (2017), pp.~965--989.

\bibitem{he_large_2007}
{\sc Y.~He, Y.~Liu, and T.~Tang}, {\em On large time-stepping methods for the
  {Cahn}-{Hilliard} equation}, Appl. Numer. Math., 57 (2007), pp.~616--628.

\bibitem{kessler_posteriori_2004}
{\sc D.~Kessler, R.~H. Nochetto, and A.~Schmidt}, {\em A posteriori error
  control for the {Allen}-{Cahn} problem: circumventing {Gronwall}'s
  inequality}, ESAIM: Math. Model. Numer. Anal., 38 (2004), pp.~129--142.

\bibitem{li_second_2017}
{\sc D.~Li and Z.~Qiao}, {\em On second order semi-implicit {Fourier} spectral
  methods for 2d {Cahn}-{Hilliard} equations}, J. Sci. Comput., 70 (2017),
  pp.~301--341.

\bibitem{li_characterizing_2016}
{\sc D.~Li, Z.~Qiao, and T.~Tang}, {\em Characterizing the stabilization size
  for semi-implicit {Fourier}-spectral method to phase field equations}, SIAM
  J. Numer. Anal., 54 (2016), pp.~1653--1681.

\bibitem{li_second_2018}
{\sc W.~Li, W.~Chen, C.~Wang, Y.~Yan, and R.~He}, {\em A second order energy
  stable linear scheme for a thin film model without slope selection}, J. Sci.
  Comput.,  (2018), pp.~1--33.

\bibitem{li_second-order_2017}
{\sc X.~Li, Z.~Qiao, and H.~Zhang}, {\em A second-order convex splitting scheme
  for a {Cahn}-{Hilliard} equation with variable interfacial parameters},
  Journal of Computational Mathematics, 35 (2017), pp.~693--710.

\bibitem{magaletti2013sharp}
{\sc F.~Magaletti, F.~Picano, M.~Chinappi, L.~Marino, and C.~M. Casciola}, {\em
  The sharp-interface limit of the {Cahn--Hilliard}/{Navier--Stokes} model for
  binary fluids}, J Fluid. Mech., 714 (2013), pp.~95--126.

\bibitem{shen_efficient_2017}
{\sc J.~Shen, J.~Xu, and J.~Yang}, {\em A new class of efficient and robust
  energy stable schemes for gradient flows}, arXiv:1710.01331,  (2017).

\bibitem{shen_SAV_2017}
\leavevmode\vrule height 2pt depth -1.6pt width 23pt, {\em The scalar auxiliary
  variable ({SAV}) approach for gradient flows}, J. Comput. Phys., 353 (2017),
  pp.~407--416.

\bibitem{shen_numerical_2010}
{\sc J.~Shen and X.~Yang}, {\em Numerical approximations of {Allen-Cahn} and
  {Cahn-Hilliard} equations}, Discrete Contin. Dyn. A., 28 (2010),
  pp.~1669--1691.

\bibitem{shen_efficient_2015}
{\sc J.~Shen, X.~Yang, and H.~Yu}, {\em Efficient energy stable numerical
  schemes for a phase field moving contact line model}, J. Comput. Phys., 284
  (2015), pp.~617--630.

\bibitem{WangYu2017b}
{\sc L.~Wang and H.~Yu}, {\em Convergence analysis of an unconditionally energy
  stable linear {Crank-Nicolson} scheme for the {Cahn-Hilliard} equation}, J.
  Math. Study, 51 (2017), pp.~89--114.

\bibitem{wang_two_2017}
{\sc L.~Wang and H.~Yu}, {\em On efficient second order stabilized
  semi-implicit schemes for the {Cahn}-{Hilliard} phase-field equation},
  arXiv:1708.09763, submitted to J. Sci. Comput.,  (2017).

\bibitem{wu_stabilized_2014}
{\sc X.~Wu, G.~J. van Zwieten, and K.~G. van~der Zee}, {\em Stabilized
  second-order convex splitting schemes for {Cahn}-{Hilliard} models with
  application to diffuse-interface tumor-growth models}, Int. J. Numer. Meth.
  Biomed. Engng., 30 (2014), pp.~180--203.

\bibitem{xu_stability_2006}
{\sc C.~Xu and T.~Tang}, {\em Stability analysis of large time-stepping methods
  for epitaxial growth models}, SIAM J. Num. Anal., 44 (2006), pp.~1759--1779.

\bibitem{xu_sharp-interface_2017}
{\sc X.~Xu, Y.~Di, and H.~Yu}, {\em Sharp-interface limits of a phase-field
  model with a generalized {N}avier slip boundary condition for moving contact
  lines}, J. Fluid Mech., to appear({arXiv}:1710.09141) (2018).

\bibitem{yan_second-order_2018}
{\sc Y.~Yan, W.~Chen, C.~Wang, and S.~Wise}, {\em A second-order energy stable
  {BDF} numerical scheme for the {C}ahn-{H}illiard equation}, Commun. Comput.
  Phys., 23 (2018), pp.~572--602.

\bibitem{yang_error_2009}
{\sc X.~Yang}, {\em Error analysis of stabilized semi-implicit method of
  {Allen}-{Cahn} equation}, Discrete. Cont. Dyn. B., 11 (2009), pp.~1057--1070.

\bibitem{yang_linear_2016}
\leavevmode\vrule height 2pt depth -1.6pt width 23pt, {\em Linear, first and
  second-order, unconditionally energy stable numerical schemes for the phase
  field model of homopolymer blends}, J. Comput. Phys., 327 (2016),
  pp.~294--316.

\bibitem{yang_efficient_2017}
{\sc X.~Yang and L.~Ju}, {\em Efficient linear schemes with unconditional
  energy stability for the phase field elastic bending energy model}, Comput.
  Method. Appl. Mech. Eng., 315 (2017-03-01), pp.~691--712.

\bibitem{yang_yu_efficient_2017}
{\sc X.~Yang and H.~Yu}, {\em Efficient second order unconditionally stable
  schemes for a phase field moving contact line model using an invariant energy
  quadratization approach}, {SIAM} J. Sci. Comput., to appear (2018).

\bibitem{yu_numerical_2017}
{\sc H.~Yu and X.~Yang}, {\em Numerical approximations for a phase-field moving
  contact line model with variable densities and viscosities}, J. Comput.
  Phys., 334 (2017), pp.~665--686.

\bibitem{yue_diffuse-interface_2004}
{\sc P.~Yue, J.~J. Feng, C.~Liu, and J.~Shen}, {\em A diffuse-interface method
  for simulating two-phase flows of complex fluids}, J. Fluid. Mech., 515
  (2004), pp.~293--317.

\bibitem{zhang_adaptive_2013}
{\sc Z.~Zhang, Y.~Ma, and Z.~Qiao}, {\em An adaptive time-stepping strategy for
  solving the phase field crystal model}, J. Comput. Phys., 249 (2013),
  pp.~204--215.

\bibitem{zhu_coarsening_1999}
{\sc J.~Zhu, L.-Q. Chen, J.~Shen, and V.~Tikare}, {\em Coarsening kinetics from
  a variable-mobility {Cahn}-{Hilliard} equation: Application of a
  semi-implicit {Fourier} spectral method}, Phys. Rev. E, 60 (1999),
  pp.~3564--3572.

\end{thebibliography}

%
%

\end{document}